\documentclass[11pt]{amsart}

\usepackage{amsmath, amssymb, amsthm, amsfonts, enumerate, color, comment}
\usepackage{mathrsfs}
\usepackage{parskip}
\usepackage{xcolor}

\usepackage{mathtools}
\mathtoolsset{showonlyrefs}

\usepackage{palatino}
\usepackage{graphicx}

\usepackage{parskip}

\numberwithin{equation}{section}
\theoremstyle{plain}
\newtheorem{Proposition}[equation]{Proposition}
\newtheorem{Corollary}[equation]{Corollary}
\newtheorem*{Corollary*}{Corollary}
\newtheorem{Theorem}[equation]{Theorem}
\newtheorem*{Theorem*}{Theorem}
\newtheorem{Lemma}[equation]{Lemma}
\theoremstyle{definition}
\newtheorem{Definition}[equation]{Definition}

\newtheorem{Example}[equation]{Example}
\newtheorem{Remark}[equation]{Remark}

\usepackage{enumitem}
\setlist[enumerate]{leftmargin=*}
\setlist[itemize]{leftmargin=*}

\setlist[enumerate,1]{label=(\alph*),font=\upshape}

\setlist[enumerate,2]{label=(\roman*),font=\upshape}

\def\C{\mathbb{C}}
\def\R{\mathbb{R}}
\def\D{\mathbb{D}}
\def\T{\mathbb{T}}
\def\N{\mathbb{N}}
\def\Z{\mathbb{Z}}

\def\K{\mathcal{K}}

\def\m{\mathcal{M}}

\newcommand{\h}{\mathcal{H}}
\newcommand{\calk}{\mathcal{K}}
\newcommand{\M}{\emph{\textbf{M}}}
\newcommand{\LL}{\emph{\textbf{L}}}
\newcommand{\II}{\mathcal{I}}
\newcommand{\U}{\mathbf{U}}

\newcommand{\CC}{\mathbf{C}}
\newcommand{\CCC}{\mathfrak{C}}
\newcommand{\f}{\emph{\textbf{f}}}
\newcommand{\J}{\emph{\textbf{J}}}
\newcommand{\I}{\emph{\textbf{I}}}

\newcommand{\A}{\emph{\textbf{A}}}

\renewcommand{\leq}{\leqslant}
\renewcommand{\geq}{\geqslant}
\renewcommand{\subset}{\subseteq}
\renewcommand{\phi}{\varphi}
\renewcommand{\vec}[1]{{\bf #1}}

\usepackage{xcolor}

\author[J. Mashreghi]{Javad Mashreghi}
\address{D\'epartement de math\'ematiques et de statistique, Universit\'e Laval, Qu\'ebec, QC,
Canada, G1K 0A6}
\email{javad.mashreghi@mat.ulaval.ca}

\author[M. Ptak]{Marek Ptak}
\address{Department of Applied Mathematics,
University of Agriculture, ul. Balicka 253c\\ 30-198 Krak\'ow, Poland.}
\email{rmptak@cyf-kr.edu.pl}

\author[W. Ross]{William T. Ross}
	\address{Department of Mathematics and Computer Science, University of Richmond, Richmond, VA 23173, USA}
	\email{wross@richmond.edu}
	

\title{Conjugations of unitary operators, I}


\thanks{This work was supported by the NSERC Discovery Grant (Canada), the Fullbright  Foundation, the Canada Research Chair program,  and by the Ministry of Science and
Higher Education of the Republic of Poland.}

\begin{document}

\begin{abstract}
If $U$ is a unitary operator on a separable complex Hilbert space $\mathcal{H}$, an application of the spectral theorem says there is a conjugation $C$ on $\mathcal{H}$ (an antilinear, involutive, isometry on $\mathcal{H}$) for which
$
C U C = U^{*}.$
 In this paper, we fix a unitary operator $U$ and describe {\em all} of the conjugations $C$  which satisfy this property. As a consequence of our results, we show that a subspace is hyperinvariant for $U$ if and only if it is invariant for any conjugation $C$ for which $CUC = U^{*}$. 
\end{abstract}

\maketitle

\maketitle

\section{Introduction}


A version of the spectral theorem says that any unitary operator $U$ on a separable complex Hilbert space $\mathcal{H}$ is unitarily equivalent to a multiplication operator $M_{\phi} f = \phi f$ on a Lebesgue space $L^2(\mu, X)$ \cite[p.~13]{MR2003221}. Here  $X$ is a compact Hausdorff space, $\mu$ is a finite positive Borel measure on $X$, and $\phi \in L^{\infty}(\mu, X)$ is unimodular $\mu$-almost everywhere. If $J$ is the mapping from $L^2(\mu, X)$ to itself defined by $(J f)(x) = \overline{f(x)}$, $x \in X$ (the bar denotes complex conjugation), then $J$ is an antilinear, isometric, and involutive map and is called a {\em conjugation}. Furthermore, the conjugation $J$  induces the  adjoint identity $J M_{\phi} J = M_{\overline{\phi}} = M_{\phi}^{*}$. Via unitary equivalence, given any unitary operator $U$ on $\h$, this results in a conjugation $C$ on $\mathcal{H}$ for which $C U C = U^{*}$ (see Lemma \ref{lem1.2}). In the parlance of operator theory, one says that $U$ is a {\em $C$-symmetric operator} \cite{MR2187654, MR2302518}. As it turns out,  there are many conjugations $C$ on $\mathcal{H}$ for which $C U C = U^{*}$. The goal of this paper is to describe them all.


Towards this goal, let
$$\mathscr{C}_s(U) := \{\mbox{$C$ is a conjugation on $\mathcal{H}$}: C U C = U^{*}\}.$$
As discussed in the previous paragraph, $\mathscr{C}_s(U) \not = \varnothing$ (see also Proposition \ref{GL}). On the other extreme, $\mathscr{C}_{s}(I)$, where $I$ is the identity operator on $\mathcal{H}$, is simply the set of {\em all} conjugations on $\h$. The subscript $s$, for ``symmetric'', in the above definition of $\mathscr{C}_{s}(U)$ might seem superfluous. However,  in a follow up paper to this one \cite{MPRCOUII}, we will explore the set $\mathscr{C}_{c}(U)$ (notice the subscript $c$), the conjugations on $\h$ that commute with $U$, i.e., $C U C = U$.
This paper provides several characterizations of $\mathscr{C}_s(U)$. One such characterization (Theorem \ref{th1.2}) involves the spectral theorem. When one is fortunate enough to have a concrete spectral decomposition of a unitary operator, for example when $U$ is an $n \times n$ unitary matrix, one can give a very tangible  description of $\mathscr{C}_s(U)$. In particular, Theorem \ref{kjahfgr4iojtegefeds00} says that if $U$ is an $n \times n$ unitary matrix
with spectral decomposition
$$U = W \begin{bmatrix}
\xi_1 I_{n_1} & & & \\
& \xi_2 I_{n_2} & & \\
& & \ddots  & \\
& & &  \xi_d I_{n_d}
\end{bmatrix}
 W^{*},$$
 where $W$ is an $n \times n$ unitary matrix, $\xi_1, \ldots, \xi_d$ ($1 \leq d \leq n$) are the distinct eigenvalues of $U$, $I_{n_j}$ is the $n_j \times n_j$ identity matrix, and $n_1 + n_2 + \cdots + n_d = n$, then any conjugation $C$ on $\C^n$ for which $C U C = U^{*}$ must take the form
 \begin{equation}\label{oOoooJJjJjjjj}
 C =
 W \begin{bmatrix}
 V_1 & & & \\
 & V_2 & & \\
 & & \ddots &\\
 & & & V_d
 \end{bmatrix} J W^{*},
 \end{equation}
  where
 each $V_j$ is an $n_j \times n_j$ unitary matrix satisfying $V_{j}^{t} = V_{j}$ ($t$ denotes the transpose of a matrix), and $J$ denotes the conjugation on $\C^n$ defined by
 $$J [x_1 \; x_2  \cdots x_n]^{t} = [\overline{x_1} \; \overline{x_2} \cdots \overline{x_n}]^{t}.$$
A version of \eqref{oOoooJJjJjjjj} can be obtained in  the infinite dimensional setting when the eigenvectors for $U$ are complete in $\h$. In particular, we give a concrete description of $\mathscr{C}_s(U)$ when $U$ is the Fourier--Plancherel  transform on $L^2(\R)$ (Example \ref{Foureoirtert}) and when $U$ is the Hilbert transform on $L^2(\R)$ (Example \ref{Hilertrt}). 

The main driver of nearly all the results in this paper is a measure-theoretic decomposition of any $C \in \mathscr{C}_s(U)$ given in Theorem \ref{conj_dec}. This enables us to examine manageable pieces of $C$ of $\h$ on certain reducing subspaces of $U$. 

Another description of $\mathscr{C}_s(U)$ relies  less on the spectral decomposition of a unitary operator $U$, which is often quite intangible in the general setting, and more on some special properties of $U$. For example, if $U$ is a {\em bilateral shift} (see Definition \ref{lafdhgjdfpsgszhnigbv888}), $\mathscr{C}_s(U)$ was described in \cite{MR4169409} (see also  Theorem \ref{jjHHbbhGHJK}). In particular, when $U$ is {\em the} bilateral shift $(M_{\xi} f)(\xi) = \xi f(\xi)$ on $L^2(m, \T)$ ($m$ is normalized Lebesgue measure on the unit circle $\T$), then $C \in \mathscr{C}_s(U)$ if and only if $C = M_{u} J$, where $J$ is the conjugation on $L^2(m, \T)$ given by $J f = \overline{f}$ and $u \in L^{\infty}(m, \T)$ is unimodular almost everywhere on $\T$ (see Example \ref{00o0iJjjIjiiJExampewl666}).

The unitary multiplication operator $M_{\psi}$ on $L^2(m, \T)$, where $\psi$ is an inner function (a bounded analytic function on $\D$ whose radial boundary values on $\T$ have modulus one $m$-almost everywhere \cite{Duren, Garnett}), turns out to be a bilateral shift \cite[Proposition 5.17]{JMMPWR_OT28} and a concrete description of $\mathscr{C}_s(M_{\xi})$ is given in Theorem \ref{089foidjgigifghjhjh77766}. In fact (see Remark \ref{s7s7s7s7s7s7s7}), such $M_{\psi}$ serve as  models for all bilateral shifts and so, in a way, Theorem \ref{089foidjgigifghjhjh77766}, is a canonical model for $\mathscr{C}_s(U)$ for any bilateral shift $U$.

As a byproduct of several results in this paper, we can connect conjugations with the hyperinvariant subspaces of a unitary operator $U$ on $\h$ (those subspace which are invariant for any bounded operator on $\h$ commuting with $U$). Certainly the invariant and hyperinvariant subspaces of unitary operators have been discussed before (for example \cite{DP, MR0253076, MR0669272}). Here we connect  them with conjugations in the following way. For a (closed) subspace $\mathcal{M}$ of $\h$, the following are equivalent: (i) $\mathcal{M}$ is hyperinvariant for $U$;  (ii) $C \mathcal{M} \subset \mathcal{M}$ for every $C \in \mathscr{C}_{s}(U)$. This result is obtained, along with several other equivalent conditions covered in this paper, via Theorems \ref{conj_dec},  \ref{s9dufioiskldfgf}, and \ref{new3} and is officially stated  in Theorem \ref{LFDKJgnsl;gdhf}.


\section{Basic facts about conjugations and spectral measures}\label{NxcvxcvxcvxXX}

{\em All Hilbert spaces $\h$ in this paper will be complex and separable.}  Let $\mathcal{B}(\h)$ denote the set of all bounded linear transformations on $\h$ and $\mathscr{A}\!\mathcal{B}(\mathcal{H})$ denote  the set of all {\em bounded antilinear transformations} on $\mathcal{H}$. By this we mean that
$C \in \mathscr{A}\!\mathcal{B}(\mathcal{H})$ when
$C(\vec{x} + \alpha \vec{y}) = C \vec{x} + \overline{\alpha} C \vec{y}$ for all $\vec{x}, \vec{y} \in \mathcal{H}$  and  $\alpha \in \C$ ($C$ is antilinear)
and  $\sup\{\|C \vec{x}\|: \|\vec{x}\| = 1\}$ is finite ($C$ is bounded).
We say that $C \in \mathscr{A}\!\mathcal{B}(\mathcal{H})$ is a {\em conjugation} if $C$ satisfies the two additional conditions $\|C \vec{x}\| = \|\vec{x}\|$ for all $\vec{x} \in \mathcal{H}$ ($C$ is isometric) and $C^2 = I$ ($C$ is involutive).
By the polarization identity, a conjugation also satisfies
\begin{equation}\label{CCCCCCcccc}
\langle C \vec{x}, C \vec{y}\rangle = \langle \vec{y}, \vec{x}\rangle \; \mbox{for all $\vec{x}, \vec{y} \in \mathcal{H}.$}
\end{equation}
Conjugations play an important role in operator theory and were initially studied in  \cite{MR2186351, MR2198373,  MR3254868,MR2187654, MR2302518}. More recently, conjugations were explored in  various settings and applications in \cite{MR4083641, MR4169409, MR4493877, MR4502709, MR4133627, MR4077555}.


\begin{Example}\label{exkjsdf77y}
Many types of conjugations  were outlined in \cite{MR3254868, MR2187654, MR2302518}. Below are a few basic ones that are relevant to this paper.
\begin{enumerate}
\item As discussed in the introduction, the mapping $C f = \bar{f}$ defines a conjugation on $L^2(\mu, X)$. In particular, the mapping $C [x_1 \; x_2 \; \cdots \; x_n]^{t} = [\overline{x_1} \; \overline{x_2} \; \cdots \; \overline{x_n}]^t$ defines a conjugation on $\C^n$. Throughout this paper we will use the notation $t$ to represent the transpose of a matrix. In addition,  vectors in $\C^n$ will be viewed  as column vectors since, for an $n \times n$ matrix $A$ of complex numbers, we will often consider the linear transformations on $\C^n$ defined by $\vec{x} \mapsto A \vec{x}$.
\item One can consider the conjugations 
$(C f)(t) = \overline{f(t)}$ and $(C f)(t) = \overline{f(-t)}$ on $L^2(\R)$.
These were used in \cite{Bender_2007, BB1998} to study symmetric operators and their connections to physics.
\item If $u$ is an inner function, $H^2$ is the standard Hardy space, and $\K_{u} := H^2 \ominus uH^2$ is the model space associated with $u$ (considering everything as a subspace of $L^2(m, \T)$ via radial boundary values), then $(C f)(\xi) = u(\xi) \overline{\xi f(\xi)}$ defines a conjugation on $\K_u$ \cite[Ch.~8]{MR3526203}.
\end{enumerate}
\end{Example}

This next lemma  enables us to transfer a conjugation on one Hilbert space to a conjugation on another. The (easy) proof is left to the reader.

\begin{Lemma}\label{lem1.2}
Suppose $\mathcal{H}$ and $\mathcal{K}$ are Hilbert spaces and $V: \mathcal{H} \to \mathcal{K}$ is a unitary operator. If $C$ is a conjugation on $\mathcal{H}$ then $VCV^{*}$ is a conjugation on $\mathcal{K}$.
\end{Lemma}



For a conjugation $C$ on $\h$ and an $A \in \mathcal{B}(\h)$, we say that $A$ is {\em $C$-symmetric} if $C A C = A^{*}$. A multitude of operators from a variety of settings enjoy this property \cite{MR3254868, MR2187654,  MR2535469, MR2661508}. Here are a few examples that are relevant to this paper.

\begin{Example}\label{9ioijHVVVGS}
\begin{enumerate}
\item If $\phi \in L^{\infty}(\mu, X)$ and $M_{\phi}$ denotes the multiplication operator $M_{\phi} f = \phi f$ on $L^2(\mu, X)$, then the conjugation $C$ from Example \ref{exkjsdf77y}(a) satisfies 
$C M_{\phi} C = M_{\overline{\phi}} = M_{\phi}^{*}.$
Thus, $M_{\phi}$ is $C$-symmetric.
 For a normal operator $N \in \mathcal{B}(\h)$, the spectral theorem says that $N$ is unitarily equivalent to $M_{\phi}$ on $L^2(\mu, X)$ for some  finite positive Borel measure on some compact Hausdorff space $X$. Now use the discussion above and Lemma \ref{lem1.2} to see that $N$ is a $C$-symmetric operator for some conjugation $C$ on $\h$.
\item The conjugation $(Cf)(x) = \overline{f(-x)}$ from Example \ref{exkjsdf77y}(b) makes the translation operator $(U f)(x) = f(x - 1)$ on $L^2(\R)$ a $C$-symmetric unitary operator. This conjugation also makes the Hilbert transform a $C$-symmetric unitary operator. 

\item The conjugation $(C f)(t) = \overline{f(t)}$ from Example \ref{exkjsdf77y}(b) makes the Fourier--Plancherel transform a $C$-symmetric unitary operator on $L^2(\R)$. 
\end{enumerate}
\end{Example}

Though not used in this paper, but certainly in our follow up paper  \cite{MPRCOUII}, we recall the following result from \cite{MR0190750} which shows that any unitary operator can be built from conjugations. We include a short proof (slightly different from the original)  for the reader's convenience.

\begin{Proposition}\label{GL} For each unitary operator $U$ on $\h$, there are conjugations $J_1$ and $J_2$ on $\mathcal{H}$ such that $U=J_1J_2$. Moreover, $J_1, J_2 \in \mathscr{C}_{s}(U)$.
\end{Proposition}

\begin{proof}
For a given unitary $U$ on $\h$, we can use the spectral theorem argument from Example \ref{9ioijHVVVGS}(a) (and also from the introduction) to produce a conjugation $J_1$ on $\h$ such that $J_1 U J_1 = U^{*}$. Now observe that $J_2 := U^{*} J_1$ is a conjugation, $J_2 U J_2 = U^{*}$, and $J_1 J_2 = U$.
\end{proof}

The following result will be used from time to time in this paper. 

\begin{Proposition}\label{bnjiovbhucfgyu}
Suppose $U, V, W$ are unitary operators on $\h$ such that $W U W^{*} = V$. Then $W \mathscr{C}_{s}(U) W^{*} = \mathscr{C}_{s}(V)$.
\end{Proposition}

The following result from  \cite[Lemma 3.2]{MR2966041}, which will be generalized below, gives an explicit description of  all the conjugations on $\C^n$.

\begin{Proposition}\label{KHADFKDJSHF}
A mapping $C$ on $\C^n$ is a conjugation  if and only if $C = V J$, where $V$ is an $n \times n$ unitary matrix with $V^{t} = V$ and $J$ is the conjugation on $\C^n$ defined by $J[x_1\; x_2\; \cdots\; x_n]^{t} = [\overline{x_1}\; \overline{x_2}\; \cdots\; \overline{x_n}]^{t}$, i.e., 
$$C[x_1 \; x_2 \; \cdots x_n]^{t} = V [\overline{x_1} \; \overline{x_2} \; \cdots \overline{x_n}]^{t}.$$
\end{Proposition}
Our discussion below needs a generalization of  the previous result for $\C^n$ to an abstract (separable) Hilbert  space $\h$, but, of course, we need a substitute for $V^{t}$ since the ``transpose'' of a linear operator $V$ on $\h$ needs a proper definition (as opposed to the adjoint $V^{*}$ which has a clear and well understood definition). For a fixed orthonormal basis $\mathscr{B} = \{\vec{u}_j\}_{j \geq 1}$ for $\h$ and an $A \in \mathcal{B}(\h)$, we use the notation 
\begin{equation}\label{kkKDSFdsf}
[A]_{\mathscr{B}} = [a_{mn}]_{m, n \geq 1} = [\langle A \vec{u}_{n}, \vec{u}_{m}\rangle]_{m, n \geq 1}
\end{equation}
 to denote the matrix representation of $A$ with respect to $\mathscr{B}$, considered as a bounded operator on the sequence space
 \begin{equation}\label{zoozZZZZellllll}
\ell^{2}_{+} := \Big\{\vec{x} = [x_1 \; x_2 \; \cdots]^{t}, x_j \in \C: \|\vec{x}\|_{\ell^{2}_{+}} = \Big(\sum_{j  \geq 1} |x_j|^2\Big)^{1/2} < \infty\Big\}
\end{equation}
by $\vec{x} \mapsto [A]_{\mathscr{B}} \vec{x}$. 
For this fixed orthonormal basis $\mathscr{B}$, we also define a conjugation $J_{\mathscr{B}}$ on $\h$ by
\begin{equation}\label{JJJBBBB}
J_{\mathscr{B}}\Big(\sum_{j \geq 1} c_j \vec{u}_{j}\Big) = \sum_{j \geq 1} \overline{c_j} \vec{u}_j.
\end{equation}
In other words, $J_{\mathscr{B}}$ fixes every basis element $\vec{u}_j$ and extends antilinearly to $\h$.
 Here is a version of Proposition \ref{KHADFKDJSHF} for a general, possibly infinite dimensional, separable Hilbert space.

\begin{Proposition}\label{9we8rouigjfledw}
For an orthonormal basis $\mathscr{B} = \{\vec{u}_j\}_{j \geq 1}$ for a separable Hilbert space  $\h$, the mapping $C$ is a conjugation on $\h$ if and only if there is a unitary operator $V$ on $\h$ such that $[V]_{\mathscr{B}}^{t} = [V]_{\mathscr{B}}$ and $C = V\!J_{\mathscr{B}}$.
\end{Proposition}

\begin{proof}
Suppose $C$ is a conjugation on $\h$. Then $V = C\!J_{\mathscr{B}}$ is a unitary operator on $\h$ (since it is linear, isometric, and onto). Moreover, applying \eqref{JJJBBBB} and then \eqref{CCCCCCcccc}, we see that 
\begin{align*}
\langle V \vec{u}_{m}, \vec{u}_{n}\rangle & = \langle C\!J_{\mathscr{B}} \vec{u}_{m}, \vec{u}_{n}\rangle\\
& = \langle C \vec{u}_{m}, \vec{u}_{n}\rangle\\
& = \langle C \vec{u}_{n}, \vec{u}_{m}\rangle\\
& = \langle C\!J_{\mathscr{B}} \vec{u}_{n}, \vec{u}_{m}\rangle\\
& = \langle V \vec{u}_{n}, \vec{u}_{m}\rangle.
\end{align*}
Thus $[V]_{\mathscr{B}}^{t} = [V]_{\mathscr{B}}$ and $C = V\!J_{\mathscr{B}}$.

Conversely, if $V$ is a unitary operator on $\h$ with $[V]_{\mathscr{B}}^{t} = [V]_{\mathscr{B}}$, define $C = V\!J_{\mathscr{B}}$. Observe that $C$ is antilinear and isometric. We just need to verify that $C^2 = I$. Since $C$ is antilinear, note that $C^2 = V\!J_{\mathscr{B}} V\!J_{\mathscr{B}}$ is a bounded {\em linear} operator on $\h$. Moreover, since $J_{\mathscr{B}} V\!J_{\mathscr{B}}$ is also a bounded linear operator on $\h$, we see that 
\begin{equation}\label{bBBbbBbBBB}
[C^2]_{\mathscr{B}} = [V]_{\mathscr{B}} [J_{\mathscr{B}} V\!J_{\mathscr{B}}]_{\mathscr{B}}.
\end{equation}
Now observe that for all $m, n \geq 1$, again making use of  \eqref{CCCCCCcccc} and \eqref{JJJBBBB},
\begin{align*}
\langle J_{\mathscr{B}} V\!J_{\mathscr{B}} \vec{u}_{m}, \vec{u}_{n}\rangle & = \langle J_{\mathscr{B}} V \vec{u}_{m}, \vec{u}_{n}\rangle\\
& = \langle J_{\mathscr{B}} \vec{u}_{n}, V \vec{u}_{m}\rangle\\
& = \langle \vec{u}_n, V \vec{u}_m\rangle\\
& = \overline{\langle V \vec{u}_{m}, \vec{u}_{n}\rangle}.
\end{align*}
Thus, from our assumption that  $[V]_{\mathscr{B}}^{t} = [V]_{\mathscr{B}}$, we see that 
\begin{equation}\label{popoPOPOPOss}
[J_{\mathscr{B}} V\!J_{\mathscr{B}}]_{\mathscr{B}} = \overline{[V]_{\mathscr{B}}} = \overline{[V]_{\mathscr{B}}^{t}} = [V^{*}]_{\mathscr{B}}.
\end{equation}
From \eqref{bBBbbBbBBB} it follows that $[C^2]_{\mathscr{B}} = [I]_{\mathscr{B}}$ and so $C^2 = I$. Thus, $C$ is a conjugation on $\h$.
\end{proof}

\begin{Remark}
If $J$ is {\em any} conjugation on $\h$, one can show  there exists an orthonormal basis $\mathscr{B}$ for $\h$ such that $J = J_{\mathscr{B}}$ \cite[Lemma 1]{MR2187654}. Thus,  there is some freedom in the above analysis, if so desired, to choose a conjugation $J$ first  instead of the orthonormal basis $\mathscr{B}$.
\end{Remark}



A version of the spectral theorem for unitary operators  (see \cite[Ch.~IX, Thm.~2.2]{ConwayFA}  or \cite[Ch.~1]{MR2003221}) says that if $U$ is a unitary operator on $\mathcal{H}$, then there is a unique spectral measure $E(\cdot)$  defined on the Borel subsets of $\sigma(U)$, the spectrum of $U$, such that
\begin{equation}\label{9sectraloth}
U = \int_{\sigma(U)} \xi d E(\xi).
\end{equation}
Moreover, for any spectral measure $E(\cdot)$ on $\T$  (i.e., $E(\cdot)$ is a projection-valued function on the Borel subsets of $\T$ such that $E(\T) = I$ and $E(\cdot)$ is countably additive), there is a unique unitary operator $U$ associated  with $E(\cdot)$ via \eqref{9sectraloth}.

 For a spectral measure $E(\cdot)$ and $\vec{x}, \vec{y} \in \mathcal{H}$, the function
$\mu_{\vec{x}, \vec{y}}(\cdot) := \langle E(\cdot) \vec{x}, \vec{y}\rangle$
 defines a finite complex Borel measure on $\T$ and, in particular, for each  $\vec{x} \in \h$,
\begin{equation}\label{67uyh87877YUYUY}
\mu_{\vec{x}}:=\mu_{\vec{x},\vec{x}}
\end{equation} defines a positive finite Borel  measure on $\T$, sometimes called an {\it elementary measure}.


The following proposition  from \cite{MR0190750} describes the relationship between a unitary operator, its spectral measure, and a conjugation.

\begin{Proposition}\label{4}
For a conjugation $C$ and a unitary operator $U$ on $\h$ with associated spectral measure $E(\cdot)$, the following are equivalent. 
\begin{enumerate}
\item $C \in \mathscr{C}_{s}(U)$; 
\item  $C E(\Omega) C = E(\Omega)$ for every Borel set $\Omega \subset \sigma(U)$.
\end{enumerate}
\end{Proposition}

\section{Decompositions of conjugations}\label{sdfourourrr}

As we will see in subsequent sections, the set $\mathscr{C}_s(U)$ is quite large and so an important step in understanding it is to decompose each $C \in \mathscr{C}_{s}(U)$  into more manageable pieces. This decomposition will involve various types of invariant subspaces. Recall that a (closed) subspace $\mathcal{M}$ of a Hilbert space $\mathcal{H}$ is {\em invariant} for $A \in \mathcal{B}(\mathcal{H})$ if $A \mathcal{M} \subset \mathcal{M}$; {\em reducing} if $A \mathcal{M} \subset \mathcal{M}$ and $A^{*} \mathcal{M} \subset \mathcal{M}$; and {\em hyperinvariant} if $T \mathcal{M} \subset \mathcal{M}$ for every $T \in \mathcal{B}(\mathcal{H})$ that commutes with $A$. We begin with the following lemma which follows from \eqref{CCCCCCcccc} and the fact that $C^2 = I$. 

\begin{Lemma}\label{nnNDF99}
If $C$ is a conjugation on $\h$ and $\mathcal{M}$ is a $C$-invariant subspace of $\h$, then $C \mathcal{M} = \mathcal{M}$ and $C\mathcal{M}^{\perp}=\mathcal{M}^\perp$.
\end{Lemma}

Below is a useful decomposition theorem for a conjugation. For the rest of  this paper,  $M_{+}(\T)$ will denote the set of all finite positive Borel measures on the unit circle $\T$. For $\mu, \sigma \in M_{+}(\T)$ use use the standard notation $\mu \ll \sigma$ for $\mu$ is absolutely continuous with respect to $\sigma$ and $\mu \perp \sigma$ for $\mu$ is singular with respect to $\sigma$.

\begin{Theorem}\label{conj_dec}Let $U$ be a unitary operator on $\h$, $E(\cdot)$ be its associated spectral measure, and for each $\vec{x} \in \h$, let $\mu_{\vec{x}}$ denote the associated elementary measure from \eqref{67uyh87877YUYUY}.  For any $\mu \in M_{+}(\T)$, let
\begin{equation}\label{kljlkjHBHH661110I}
 \h_\mu :=\{\vec{x} \in \h: \mu_{\vec{x}} \ll \mu \}.
\end{equation}
Then we have the following.
\begin{enumerate}
\item $\h_\mu$ is a reducing subspace for $U$.
  \item{}  If $C \in \mathscr{C}_s(U)$, then
      \begin{enumerate}
                  \item $C\h_\mu=\h_\mu$ and $C\h_\mu^{\perp}=\h_\mu^{\perp}$.
         \item $C=C_\mu\oplus C_{\mu}^{\perp}$ where $C_\mu := C|_{\h_{\mu}}$ and $ C_{\mu}^{\perp}: = C|_{\h_{\mu}^{\perp}}$.
         \item $C_{\mu} (U|_{\mathcal{H}_{\mu}}) C_{\mu} = U^{*}|_{\mathcal{H}_{\mu}}$ and $C_{\mu}^{\perp} (U|_{\mathcal{H}_{\mu}^{\perp}} )C_{\mu}^{\perp} = U^{*}|_{\mathcal{H}_{\mu}^{\perp}}.$
     \end{enumerate}
     \end{enumerate}
\end{Theorem}

\begin{proof}
The proof of (a) is routine and we omit it (see \cite[\S 65, \S 66]{MR0045309} for the details).

To prove (b), let $\vec{x}\in\h_\mu$ and $\Omega$ be a Borel subset of $\sigma(U)$ for which $\mu(\Omega) = 0$.
By Proposition \ref{4}, $C$ commutes with $E(\Omega)$ and thus, via \eqref{CCCCCCcccc},
\begin{align*}
\mu_{C \vec{x}} = \langle E(\Omega)C\vec{x},C\vec{x}\rangle  =\langle CE(\Omega)\vec{x},C\vec{x}\rangle
= \langle \vec{x},E(\Omega)\vec{x}\rangle
 =\mu_{\vec{x}}(\Omega)
 = 0.
\end{align*}
This proves that  $C\vec{x}\in\h_\mu$. Since $C$ is a conjugation, we see that $C\h_\mu=\h_\mu$ and  $C\h_\mu^\perp = \h_\mu^\perp$ (Lemma \ref{nnNDF99}). This proves (i) and (ii). Part (iii) follows from the facts that $\h_{\mu}$ is a reducing subspace for $U$ and $C \in \mathscr{C}_{s}(U)$.
\end{proof}

As a note here, other properties of the reducing subspace $\h_{\mu}$ were explored by Halmos in \cite[\S 65]{MR0045309}. Observe that if $\alpha \in \T$ and $\delta_{\alpha}$ denotes the point mass at $\alpha$, one can see that 
$$\h_{\delta_{\alpha}} = \ker (U - \alpha I).$$

\begin{Corollary}\label{oerioijJHJHJHssdf}
For a unitary operator  $U$ on $\h$, let $C \in \mathscr{C}_s(U)$.
  Suppose $\alpha \in \mathbb{T}$ is an eigenvalue for $U$.
  Then
  $C=C_{\delta_\alpha}\oplus C_{\delta_\alpha}^\perp,$ where $C_{\delta_\alpha} \in \mathscr{C}_{s}(U|_{\ker(U - \alpha I)})$ and $ C_{\delta_\alpha}^\perp \in \mathscr{C}_{s}(U|_{\ker(U-\alpha I)^\perp})$.
  \end{Corollary}

The following corollary will play an important role later on. 

  \begin{Corollary}\label{eigennenenen}
If $U$ is a unitary operator on $\h$ with
  $$\h = \bigoplus_{j \geq 1} \mathcal{E}_{\xi_j}, \quad \mathcal{E}_{\xi_j} := \ker(U - \xi_{j} I),$$ then
  \begin{equation}\label{conjuosidufiuiuCdelta}
  C = \bigoplus_{j \geq 1} C|_{\mathcal{E}_{\xi_j}}
  \end{equation}
  and
  \begin{equation}\label{decomposeeee}
  C|_{\mathscr{E}_{\xi_j}} U|_{\mathscr{E}_{\xi_j}} C|_{\mathscr{E}_{\xi_j}} = U^{*}|_{\mathscr{E}_{\xi_j}} \; \mbox{for all $j \geq 1$}.
  \end{equation}

  \end{Corollary}

\begin{Remark}
Apply Theorem \ref{conj_dec} to the case when $\mu = m$ (normalized Lebesgue measure on $\T$) and let $\h_{\operatorname{ac}}: =\h_m$ and $\h_{\operatorname{sing}}: =\h_{m}^\perp$. The decomposition  $\mathcal{H} = \h_{\operatorname{ac}} \oplus \h_{\operatorname{sing}}$ was  first explored in \cite{MR0045309} and considered further in \cite{MR3192030, JMMPWR_OT28, MR0285914}. One can prove the following:  Let $U$ be a unitary operator on $\mathcal{H}$ and let $C \in \mathscr{C}_s(U)$.
  If $\h=\h_{\operatorname{ac}}\oplus \h_{\operatorname{sing}}$ is the Lebesgue decomposition from above,  then  $C=C_{\operatorname{ac}}\oplus C_{\operatorname{sing}}$, where $C_{\operatorname{ac}} \in \mathscr{C}_{s}(U|_{\h_{\operatorname{ac}}})$ and   $C_{\operatorname{sing}} \in \mathscr{C}_{s}(U|_{\h_{\operatorname{sing}}})$.

A von Neumann--Wold type decomposition surveyed in  \cite{JMMPWR_OT28} yields the following decomposition of a unitary operator $U$ on $\mathcal{H}$ by reducing subspaces $\mathcal{H}_{\operatorname{ac}}$, $\mathcal{L}$, and $\bigoplus_{\alpha \in \T} \ker(U - \alpha I)$ such that
\begin{equation}\label{88uuu8hhc8c8c}
\mathcal{H} = \mathcal{H}_{\operatorname{ac}} \bigoplus \mathcal{L} \bigoplus_{\alpha \in \T} \ker(U - \alpha I).
\end{equation}
Here one proves the following: 
Let $U$ be a unitary operator on $\mathcal{H}$ and $C \in \mathscr{C}_s(U)$.
 If $\h$ is decomposed as in \eqref{88uuu8hhc8c8c},   then there is an appropriate decomposition of $C$ as $C=C_{ac}\oplus C|_{\mathcal{L}}\oplus\bigoplus_{\alpha\in \mathbb{T}} C_{\delta_\alpha}$, where each summand is a conjugation on the corresponding space.
\end{Remark}

Though Theorem \ref{conj_dec} seems relatively simple, it yields interesting results about hyperinvariant subspaces of unitary operators.
For a unitary operator, every hyperinvariant subspace is also reducing (but not the converse). Theorem \ref{s9dufioiskldfgf} below connects hyperinvariant subspaces with conjugations. The culmination  of this discussion will be given in Theorems \ref{new3} and \ref{LFDKJgnsl;gdhf}.

\begin{Theorem}\label{iuyrituyeriutyhype}
Let $U$ be a unitary operator on $\h$ with spectral measure $E(\cdot)$. Then every hyperinvariant subspace for $U$ is invariant for every $C \in \mathscr{C}_{s}(U)$.
 Moreover, every hyperinvariant subspaces can be written as  $E(\Omega)\h$, where $\Omega$ is a Borel subset of $ \sigma(U)$.
 \end{Theorem}

\begin{proof}
From \cite[Proposition 6.9]{MR2003221} (see also \cite{DP}), every hyperinvariant subspace for $U$ can be written as $E(\Omega)\h$ for some Borel set $\Omega\subset\sigma(U)$.  Since $C$ commutes with $E(\cdot)$ (Proposition \ref{4}(b)), $E(\Omega)\h$ is invariant for $C$.
\end{proof}


We now connect $\h_{\mu}$ with hyperinvariant subspaces. A first step in this direction is this following result -- which might be contained somewhere in the literature but we include a proof for the reader's convenience. We will see an extension of this result in Theorems \ref{new3} and \ref{LFDKJgnsl;gdhf} below.


\begin{Theorem}\label{s9dufioiskldfgf}
If $\mathcal{M}$ is a hyperinvariant subspace of a unitary operator $U$ on $\h$, then there is a $\mu \in M_{+}(\T)$ for which $\mathcal{M} = \h_{\mu}$.
\end{Theorem}

\begin{proof}
From Theorem \ref{iuyrituyeriutyhype}, every hyperinvarient subspace $\mathcal{M}$ of $U$, with associated spectral measure $E(\cdot)$, takes the form $\mathcal{M} = E(\Omega) \h$ for some Borel subset $\Omega \subset \sigma(U)$. Since $E(\Omega)\h$ is also reducing, then $V := U|_{E(\Omega) \h}$ is a unitary operator whose spectral measure is $E(\cdot)E(\Omega)$.

Let $\mathcal{W}^{*}(V)$ denote the von Neumann algebra generated by $V$ (i.e., the smallest $\ast$-closed and strongly closed algebra containing $V$). Using \cite[Ch.~IX, Cor.~7.9]{ConwayFA} one can produce a {\em separating vector} $\vec{e} \in \h$ for $\mathcal{W}^{*}(V)$ in that if $A \in \mathcal{W}^{*}(V)$ satisfies $A \vec{e} = \vec{0}$, then $A \equiv 0$. By \cite[Ch.~IX, Prop. 8.3]{ConwayFA}, the corresponding elementary measure $\mu_{\vec{e}}$ from \eqref{67uyh87877YUYUY} is a  {\em scalar-valued spectral measure} for $V$ in that $\mu_{\vec{e}}(\Delta) = 0$ if and only if $E(\Delta) E(\Omega)  = E(\Delta \cap \Omega)= 0$.

To complete the proof, we will show that $E(\Omega) \h = \h_{\mu_{\vec{e}}}$. For the $\subset$ containment, observe that if $\vec{x} \in E(\Omega) \h$ and $\Delta$ is a Borel subset of $\Omega$, then
\begin{equation}\label{dkjfblfdkjgkf444}
\mu_{\vec{x}}(\Delta) = \langle E(\Delta) \vec{x}, \vec{x}\rangle = \|E(\Delta) \vec{x}\|^2
\end{equation}
 and thus, since $\mu_{\vec{e}}$ is a scalar valued spectral measure for $V = U|_{E(\Omega) \h}$, we see that if $\mu_{\vec{e}}(\Delta) = 0$, then $E(\Delta) = 0$. Thus, by \eqref{dkjfblfdkjgkf444}, $\mu_{\vec{x}} \ll \mu_{\vec{e}}$ and so $\vec{x} \in \h_{\mu_{\vec{e}}}$. This verifies $E(\Omega) \h \subset \h_{\mu_{\vec{e}}}$.

For the $\supseteq$ containment, let $\vec{y} \in \h_{\mu_{\vec{e}}}$. Then $\mu_{\vec{y}} \ll \mu_{\vec{e}}$ and so $\mu_{\vec{y}}(\Delta) = 0$ whenever $\mu_{\vec{e}}(\Delta) = 0$. We want to show that $\vec{y} \in E(\Omega) \h$. Since $\mu_{\vec{e}}$ is a scalar valued spectral measure for $V$, we see that $\mu_{\vec{e}}(\sigma(U) \setminus \Omega) = 0$. Using the assumption that $\mu_{\vec{y}} \ll \mu_{\vec{e}}$, we obtain
$$\mu_{\vec{y}}(\sigma(U) \setminus \Omega) = \|E(\sigma(U) \setminus \Omega) \vec{y}\|^2 = 0.$$
Thus, since $\h = (E(\Omega)\h) \oplus (E(\sigma(U) \setminus \Omega)\h)$, it follows that $\vec{y} \in E(\Omega) \h$. This verifies $\h_{\mu_{\vec{e}}} \subset E(\Omega) \h$.
\end{proof}

We end this section with a discussion of when $\h_{\nu_1} \subset \h_{\nu_2}$ for $\nu_1, \nu_2 \in M_{+}(\T)$. Observe that if $\nu$ is a scalar spectral measure for $U$ (i.e., $\nu(\Delta) = 0$ if and only if $E(\Delta) = 0$) then $\h_\nu=\h$. Next let
$$[\vec{x}]_{U, U^{*}} := \bigvee\{ U^{n} \vec{x}: n \in \Z\},$$
where $\bigvee$ denotes the closed linear span,
denote the $\ast$-cyclic invariant subspace generated by $\vec{x}$. 
For any  $\mu \in M_{+}(\T)$ the condition $\vec{x}\in \h_\mu$ implies that  $[\vec{x}]_{U, U^{*}}\subset \h_\mu$ (since $\h_{\mu}$ is reducing).

Recall \cite[\S 48]{MR0045309} the  standard Boolean operations $\wedge$ and $\vee$ for $\mu_1, \mu_2 \in M_{+}(\T)$ defined on Borel subsets $\Omega$ of $\T$ by 
$$(\mu_1 \vee \mu_2) (\Omega) = \mu_1(\Omega) + \mu_2(\Omega);$$
$$(\mu_1 \wedge \mu_2)(\Omega) = \inf\{\mu_1(\Omega \cap A) + \mu_2(\Omega \setminus A): \mbox{$A$ is a Borel set}\}.$$

\begin{Proposition}Let $U$ be a unitary operator on $\h$ and $\mu$ be any scalar spectral measure for $U$. For  $\nu_1, \nu_2 \in M_{+}(\T)$ the following are equivalent. 
\begin{enumerate}
\item $\h_{\nu_1}\subset \h_{\nu_2};$
\item $\nu_1\wedge \mu \ll \nu_2\wedge \mu.$
\end{enumerate}
\end{Proposition}

\begin{proof} Assume condition (b) and let $\vec{x}\in \h_{\nu_1}$. Then $\langle E(\cdot)\vec{x},\vec{x}\rangle\ll  \nu_1$. Since $\mu$ is a scalar spectral measure for $U$, we see that $\langle E(\cdot)\vec{x},\vec{x}\rangle\ll  \mu$. Hence $\langle E(\cdot)\vec{x},\vec{x}\rangle\ll  \nu_1\wedge \mu \ll \nu_2\wedge \mu$. Thus $\langle E(\cdot)\vec{x},\vec{x}\rangle\ll \nu_2$
 and so $\vec{x}\in \h_{\nu_2}$. Thus, $(b) \Longrightarrow (a)$.
 
For the proof of $(a) \Longrightarrow (b)$,  let us assume that $\nu_1\wedge \mu$ is not absolutely continuous with  respect to $\nu_2\wedge \mu$. Then there is a nonzero measure $\nu \in M_{+}(\T)$ such that $\nu\ll\nu_1\wedge \mu$ and $\nu\perp \nu_2\wedge \mu$. Let $\vec{e}\in\h$ be a vector  such that $\mu_{\vec{e}}$ is a scalar spectral measure (for example a separating vector for von Neumann algebra generated by $U$ -- from the proof of Theorem \ref{s9dufioiskldfgf}). Since $\mu$ and  $\mu_{\vec{e}}$ are mutually absolutely continuous as scalar  spectral measures for $U$,  it follows that $\nu\ll\mu_{\vec{e}}$. By \cite[Ch.IX, Lemma 8.6]{ConwayFA} there is a nonzero  $\vec{y}\in [\vec{x}]_{U, U^{*}}$
such that $\mu_{\vec{y}}=\nu$. Therefore, $\vec{y}\in \h_{\nu_1}$. On the other hand if  $\vec{y}\in \h_{\nu_2}$ then $\mu_{\vec{y}}=\nu\ll\nu_2\wedge \mu$, which is a contradiction.
  \end{proof}
  
  \begin{Corollary}
  Let $\nu_1, \nu_2 \in M_{+}(\T)$. If $\nu_1 \ll \nu_2$ then $\h_{\nu_1} \subset \h_{\nu_2}$.
  \end{Corollary}

\begin{Corollary}Let $U$ be a unitary operator  on $\h$ and $\mu$ be any scalar spectral measure for $U$. For $\nu_1, \nu_2 \in M_{+}(\T)$ the following are equivalent. 
\begin{enumerate}
\item $\h_{\nu_1}= \h_{\nu_2}$;
\item $ \nu_1\wedge \mu$ and $\nu_2\wedge \mu$ are mutually absolutely continuous.
\end{enumerate}
\end{Corollary}

\section{A matrix description of $\mathscr{C}_s(U)$}

Our first description of $\mathscr{C}_s(U)$ will involve matrix a representation of $U$ with respect to an orthonormal basis. Though this might seem a bit difficult to apply at first, we will see, through our decomposition theorem from the previous section, that it can often yield a tangible description of $\mathscr{C}_s(U)$. Recall our notation from \S \ref{NxcvxcvxcvxXX}, where, for a fixed orthonormal basis $\mathscr{B} = \{\vec{u}_{j}\}_{j \geq 1}$ for a (separable) Hilbert space $\h$ and for $A \in \mathcal{B}(\h)$, $[A]_{\mathscr{B}}$ denotes the matrix representation of $A$ with respect to $\mathscr{B}$, and $J_{\mathscr{B}}$ denotes the conjugation on $\h$ for which $J_{\mathscr{B}} \vec{u}_n = \vec{u}_{n}$ for all $n \geq 1$.

\begin{Theorem}\label{matrixedecidsfs98}
Suppose $U$ is a unitary operator on $\h$. Then $C \in \mathscr{C}_s(U)$ if and only if there is a unitary operator $V$ on $\h$ satisfying 
\begin{enumerate}
\item $[V]_{\mathscr{B}}^{t} = [V]_{\mathscr{B}}$;
\item $[V]_{\mathscr{B}} [U]_{\mathscr{B}}^{t} = [U]_{\mathscr{B}} [V]_{\mathscr{B}}$;
\item $C = V\!J_{\mathscr{B}}$.
\end{enumerate}
\end{Theorem}

\begin{proof}
Proposition \ref{9we8rouigjfledw} says that $C$ is a conjugation on $\h$ if and only if  $C = V\!J_{\mathscr{B}}$ for some unitary $V$ on $\h$ with $[V]_{\mathscr{B}}^{t} = [V]_{\mathscr{B}}$. Now, to fulfill the requirement that $C \in \mathscr{C}_{s}(U)$,  observe that
\begin{align*}
 C U^{*} C = U
& \iff (V\!J_{\mathscr{B}}) U^{*} (V\!J_{\mathscr{B}} ) = U\\
& \iff V (J_{\mathscr{B}} U^{*}\!J_{\mathscr{B}}) (J_{\mathscr{B}} V\!J_{\mathscr{B}}) = U\\
& \iff [V]_{\mathscr{B}} [J_{\mathscr{B}} U^{*}\!J_{\mathscr{B}}]_{\mathscr{B}} [J_{\mathscr{B}} V  \!J_{\mathscr{B}} ]_{\mathscr{B}} = [U]_{\mathscr{B}}\\
& \iff [V]_{\mathscr{B}} [U]_{\mathscr{B}}^{t} [V^{*}]_{\mathscr{B}} = [U]_{\mathscr{B}}.
\end{align*}
Notice the use of \eqref{popoPOPOPOss} above. Thus,
$$C U^{*} C = U \iff [V]_{\mathscr{B}} [U]_{\mathscr{B}}^{t} = [U]_{\mathscr{B}} [V]_{\mathscr{B}},$$
which completes the proof.
\end{proof}

One can use the above discussion to describe $\mathscr{C}_s(U)$ when $U$ is an $n \times n$ unitary matrix.

\begin{Theorem}\label{kjahfgr4iojtegefeds00}
Suppose $U$ is an $n \times n$ unitary matrix with spectral decomposition
$$U = W \begin{bmatrix}
\xi_1 I_{n_1} & & & \\
& \xi_2 I_{n_2} & & \\
& & \ddots  & \\
& & &  \xi_d I_{n_d}
\end{bmatrix}
 W^{*},$$
 where $W$ is an $n \times n$ unitary matrix, $\xi_1, \ldots, \xi_d \in \T$ are the distinct eigenvalues of $U$, $I_{n_j}$ is the $n_j \times n_j$ identity matrix, and $n_1 + n_2 + \cdots + n_d = n$. Then $C \in \mathscr{C}_{s}(U)$ if and only if 
 \begin{equation}\label{CCCococoCccc}
 C =
 W \begin{bmatrix}
 V_1 & & & \\
 & V_2 & & \\
 & & \ddots &\\
 & & & V_d
 \end{bmatrix} 
J W^{*},
\end{equation} where, for each $1 \leq j \leq d$, $V_j$ is an $n_j \times n_j$ unitary matrix  which satisfies $V_{j}^{t} = V_j$, and $J$ is the conjugation on $\C^n$ defined by 
 $$J[x_1 \; x_2 \; \cdots \; x_n]^{t} = [\overline{x_1} \;  \overline{x_2} \; \cdots \; \overline{x_n}]^{t}.$$ 
 \end{Theorem}

\begin{proof}
Suppose that $C \in \mathscr{C}_{s}(U)$. 
For the unitary matrix 
$$\widetilde{U} =  \begin{bmatrix}
\xi_1 I_{n_1} & & & \\
& \xi_2 I_{n_2} & & \\
& & \ddots  & \\
& & &  \xi_d I_{n_d}
\end{bmatrix},$$
Corollary \ref{eigennenenen} says that any $\widetilde{C} \in \mathscr{C}_{s}(\widetilde{U})$ can be written as 
$$\widetilde{C} = \widetilde{C}|_{\mathcal{E}_{\xi_{1}}} \oplus \widetilde{C}|_{\mathcal{E}_{\xi_{2}}}  \oplus \cdots \oplus \widetilde{C}|_{\mathcal{E}_{\xi_{d}}},$$
where 
$$\mathscr{E}_{j} = \ker (\widetilde{U} - \xi_j  I)$$
(which is the span of appropriate standard basis vectors for $\C^n$)
and
$$\widetilde{C}_{j}:= \widetilde{C}|_{\mathcal{E}_{\xi}} \in \mathscr{C}_s( \widetilde{U}|_{\mathscr{E}_{\xi_j}})  \; \mbox{for all $1 \leq j \leq d$}.$$
 Theorem \ref{matrixedecidsfs98} says that for each $1 \leq j \leq d$  there is an $n_j \times n_j$ unitary matrix $V_j$ such that such that $\widetilde{C}_{j} = V_{j} J|_{\mathscr{E}_{j}}$ and $V_{j}^{t} = V_j$. Note that $J|_{\mathscr{E}_j}$ fixes the appropriate standard basis vectors. Since $\widetilde{U}|_{\mathcal{E}_{\xi_j}} = \xi_j I_{n_j}$, condition (b) of Theorem \ref{matrixedecidsfs98} is automatic. Thus, every $\widetilde{C} \in \mathscr{C}_{s}(\widetilde{U})$ takes the form 
 $$ \begin{bmatrix}
 V_1 & & & \\
 & V_2 & & \\
 & & \ddots &\\
 & & & V_d
 \end{bmatrix} 
J.$$
Now apply Lemma \ref{lem1.2} and Proposition \ref{bnjiovbhucfgyu} to see that $C$ takes the form in \eqref{CCCococoCccc}.

Conversely suppose that $C$ is the conjugation from \eqref{CCCococoCccc}. Then, due to the fact that each of the unitary matrices $V_j$ satisfy  $V_{j}^{t} = V_j$ and  $J A J = \overline{A}$ for any matrix $A$,  $C U C$ is equal to 
\begin{align*}
 &  W  \begin{bmatrix}
 V_1 & & & \\
 & V_2 & & \\
 & & \ddots &\\
 & & & V_d
 \end{bmatrix} J \begin{bmatrix}
\xi_1 I_{n_1} & & & \\
& \xi_2 I_{n_2} & & \\
& & \ddots  & \\
& & &  \xi_d I_{n_d}
\end{bmatrix}
 \begin{bmatrix}
 V_1 & & & \\
 & V_2 & & \\
 & & \ddots &\\
 & & & V_d
 \end{bmatrix} J W^{*} \\
& = W  \begin{bmatrix}
 V_1 & & & \\
 & V_2 & & \\
 & & \ddots &\\
 & & & V_d
 \end{bmatrix} J \begin{bmatrix}
\xi_1  V_1 & & & \\
& \xi_2  V_2 & & \\
& & \ddots  & \\
& & &  \xi_d V_d
\end{bmatrix}
J W^{*}\\
& = W \begin{bmatrix}
 V_1 & & & \\
 & V_2 & & \\
 & & \ddots &\\
 & & & V_d
 \end{bmatrix} \begin{bmatrix}
\overline{\xi_1} \overline{V_1} & & & \\
&\overline{ \xi_2} \overline{ V_2} & & \\
& & \ddots  & \\
& & & \overline{ \xi_d} \overline{ V_d}
\end{bmatrix} W^{*}\\
& = W \begin{bmatrix}
 V_1 & & & \\
 & V_2 & & \\
 & & \ddots &\\
 & & & V_d
 \end{bmatrix} \begin{bmatrix}
\overline{\xi_1} V_{1}^{*} & & & \\
&\overline{ \xi_2} V_{2}^{*} & & \\
& & \ddots  & \\
& & & \overline{ \xi_d} V_{d}^{*}
\end{bmatrix} W^{*}\\
& =  W  \begin{bmatrix}
\overline{\xi_1} I_{n_1} & & & \\
&\overline{ \xi_2} I_{n_2} & & \\
& & \ddots  & \\
& & & \overline{ \xi_d} I_{d_d}
\end{bmatrix} W^{*} = U^{*}.
\end{align*}
Thus $C \in \mathscr{C}_{c}(U)$, which completes the proof. 
\end{proof}

In certain circumstances, we can extend Theorem \ref{kjahfgr4iojtegefeds00} to the infinite dimensional setting. We mention two interesting examples of unitary operators from harmonic analysis.

\begin{Example}\label{Foureoirtert}
Let 
  $$(\mathcal{F}f)(x) = \frac{1}{\sqrt{2 \pi}} \int_{-\infty}^{\infty} f(t) e^{-ixt} dt,$$ denote the classical Fourier--Plancherel transform on $L^2(\R)$ (defined initially on $L^1(\R) \cap L^2(\R)$ by the above integral and extended to be a unitary operator on $L^2(\R)$). It is known that  $\sigma(\mathcal{F}) = \sigma_{p}(\mathcal{F}) = \{1, -1, i, -i\}$ and $\mathcal{F} H_n = (-i)^n H_n$,
where $H_n$ is the $n$th Hermite function, i.e., $$H_{n}(x) = c_n e^{-x^2/2} h_{n}(x), \; \; n \geq 0,$$ 
 $c_n$ is a constant for which $\|H_n\|_{L^2(\R)}= 1$, and $h_n$ is the $n$-th Hermite polynomial.
 Moreover, $\mathscr{B} = \{H_{n}\}_{n \geq 0}$ forms an orthonormal basis for $L^2(\R)$ \cite[Ch.11]{MR4545809}.
Thus, using our earlier notation from Corollary \ref{eigennenenen},
\begin{equation}\label{77y71199UU}
L^2(\R) = \mathcal{E}_{1} \bigoplus \mathcal{E}_{-i} \bigoplus \mathcal{E}_{-1} \bigoplus \mathcal{E}_{i},
\end{equation} where
$\mathscr{B}_{1} = \{H_{4 k}\}_{k \geq 0},$
$\mathscr{B}_{-i} = \{H_{4 k + 1}\}_{k \geq 0},$
$\mathscr{B}_{-1} = \{H_{4 k + 2}\}_{k \geq 0},$
$\mathscr{B}_{i} = \{H_{4 k + 3}\}_{k \geq 0},$ are orthonormal bases for $\mathcal{E}_{1},  \mathcal{E}_{-i}, \mathcal{E}_{-}, \mathcal{E}_{-i}$ respectively, we can write $\mathcal{F}$ in block operator form with respect to the decomposition in \eqref{77y71199UU} as
$$\mathcal{F} =
\begin{bmatrix}
I|_{\mathcal{E}_{1}} & & & \\
& -i I|_{\mathcal{E}_{-i}} & & \\
& & -I|_{\mathcal{E}_{-1}} & \\
& & & i I|_{\mathcal{E}_{i}}
\end{bmatrix}.$$
Furthermore,  any conjugation $C$ on $L^2(\R)$ for which $C \mathcal{F} C = \mathcal{F}^{*}$ must take the (block)  form
$$C = \begin{bmatrix}
V_1 & & &\\
& V_{-i} & &\\
& & V_{-1} &\\
& & & V_i\\
\end{bmatrix}
\begin{bmatrix}
J_1 & & &\\
& J_{-i} & &\\
& & J_{-1} &\\
& & & J_{i}\\
\end{bmatrix},$$
where, for each $k \in \{1, -i, -1, i\}$, $J_{k}$ is the conjugation on $\mathcal{E}_{k}$ which fixes every element of $\mathscr{B}_{k}$ and $V_{k}$ is a unitary operator on $\mathcal{E}_{k}$ for which $[V_k]_{\mathscr{B}_k}^{t} = [V_k]_{\mathscr{B}_k}$.
\end{Example}

\begin{Example}\label{Hilertrt}
Suppose
 $$(\mathscr{H} g)(x) = \frac{1}{\pi}  \operatorname{PV} \int_{-\infty}^{\infty} \frac{g(t)}{x - t} dt,$$ is the Hilbert transform on $L^2(\R)$. Then $\sigma(\mathscr{H}) = \sigma_{p}(\mathscr{H}) = \{i, -i\}$ and, again using the notation from Corollary \ref{eigennenenen},
$L^2(\R) = \mathcal{E}_{i} \bigoplus \mathcal{E}_{-i}.$
Moreover, $\mathcal{E}_{i}$ has orthonormal basis $\mathscr{B}_{i} = \{f_n\}_{n \geq 1}$, where
$$f_n(x) = \frac{1}{\sqrt{\pi}} \frac{(x + i)^{n - 1}}{(x - i)^{n}},$$
and $\mathcal{E}_{-i}$ has orthonormal basis $\mathscr{B}_{-i} = \{g_n\}_{n \geq 0}$, where 
$$g_{n}(x) = \frac{1}{\sqrt{\pi}} \frac{(x - i)^{n}}{(x + i)^{n + 1}}.$$
See \cite[Ch. 12]{MR4545809} for the details.
Moreover, by a similar discussion as in the previous example, any conjugation $C$ on $L^2(\R)$ for which $C \mathscr{H} C = \mathscr{H}^{*}$ must take the  (block) form
$$C =
\begin{bmatrix}
V_i & \\
& V_{-i}
\end{bmatrix}
\begin{bmatrix}
J_{i} & \\
& J_{-i}
\end{bmatrix},$$
where for each $k \in \{i, -i\}$, $J_{k}$ is the conjugation on $\mathcal{E}_{k}$ which fixes every element of $\mathscr{B}_{i}$ and $V_{k}$ is a unitary operator on $\mathcal{E}_{k}$ for which $[V_k]_{\mathscr{B}_{k}}^{t} = [V_k]_{\mathscr{B}_{k}}$.
\end{Example}

    \section{Natural conjugations on vector valued $L^2$ spaces}\label{section3}

 This section provides a model for conjugations on vector valued Lebesgue spaces. It will be useful for our description of $\mathscr{C}_s(U)$ in \S \ref{ST} and also sets up our discussion of models for bilateral shifts, and their associated conjugations, in the next section.

For a Hilbert space $\h$ with norm $\|\cdot\|_{\h}$ and $\mu \in M_{+}(\T)$, consider the set $\mathscr{L}^{0}(\mu, \h)$ of all $\h$-valued $\mu$-measurable functions $\f$ on $\T$ and the Hilbert space
$$\mathscr{L}^{2}(\mu, \h) := \Big\{\f \in \mathscr{L}^{0}(\mu, \h): \|\f\|_{L^2(\mu. \h)} :=  \Big(\int_{\T} \|\f(\xi)\|^{2}_{\h}d\mu(\xi)\Big)^{1/2} < \infty\Big\}.$$ 
One often sees this using tensor notation as $L^2(\mu) \otimes \h$. Also consider $\mathscr{L}^{\infty}(\mu, \mathcal{B}(\h))$, the $\mu$-essentially bounded $\mathcal{B}(\h)$-valued functions $\U$ on $\T$. For $\U \in \mathscr{L}^{\infty}(\mu, \mathcal{B}(\h))$, define the multiplication operator $\M_{\U}$ on $\mathscr{L}^2(\mu, \h)$ by
\begin{equation}\label{e1.4}
  (\M_{\U}\,\f)(\xi)=\U(\xi)\f(\xi)
\end{equation}
for  $\f\in \mathscr{L}^2(\mu,\h)$ and $\mu$-almost every $\xi \in \T$.
Clearly $\M_{\U}\in \mathcal{B}(\mathscr{L}^2(\mu,\h))$. If we use the notation $\U^{*}(\xi) = \U(\xi)^{*}$, one can verify that
$$\M_{\U}^{*} = \M_{\U^{*}}.$$ We let $L^{\infty}(\mu) := \mathscr{L}^{\infty}(\mu, \C)$ to denote the set of all  scalar valued $\mu$-essentially bounded functions on $\T$. For ease of notation, ee will write $\M_{\varphi}$, when $\phi \in L^{\infty}(\mu)$, in place of the more cumbersome $\M_{\varphi \I_{\h}}$, that is,
\begin{equation}\label{e1.5}
  (\M_\varphi \f)(\xi)=(\M_{\varphi \I_\h}\f)(\xi)=\varphi(\xi)\f(\xi)
\end{equation}
for $\f \in \mathscr{L}^2(\mu, \h)$ and for $\mu$-almost every $\xi \in \T$. Of particular importance is when  $\phi(\xi) = \xi$ which yields the (bilateral) shift $\M_{\xi}$ on $\mathscr{L}^2(\mu, \h)$. As a specific example, we have {\em the} bilateral shift $(M_{\xi} f)(\xi) = \xi f(\xi)$ on $L^2(m, \T)$ when $\h = \C$.

Recall from \S \ref{NxcvxcvxcvxXX} that $\mathscr{A}\!\mathcal{B}(\h)$ denotes the space of all bounded antilinear operators on $\h$ . By $\mathscr{L}^{\infty}(\mu, \mathscr{A}\!\mathcal{B}(\h))$ we denote the space of  all $\mu$-essentially bounded and $\mathscr{A}\!\mathcal{B}(\h)$-valued Borel functions on $\T$. Analogously, as in \eqref{e1.4}, for $\CC\in \mathscr{L}^{\infty}(\mu, \mathscr{A}\!\mathcal{B}(\h))$, define
\begin{equation*}
  (\A_\CC \f)(\xi)=\CC(\xi)\f(\xi)
\end{equation*}
for $\f\in \mathscr{L}^2(\mu,\h)$ and for $\mu$-almost every $\xi \in \T$.
One can check that   $\A_\CC\in \mathscr{A}\!\mathcal{B}(\mathscr{L}^2(\mu,\h))$.

For any conjugation $J$ on $\h$, define the associated  conjugation $\J$ on $\mathscr{L}^2(\mu,\h)$ by
\begin{equation}\label{e1.6}
  (\J\f)(\xi)=J(\f(\xi)).
  \end{equation}
  For example, when $\h = \C$ and $J z = \overline{z}$ on $\C$, then $\J$ is the conjugation on $L^2(\mu)$ defined by $\J f = \overline{f}$ and discussed earlier. 
One can check that for  $\f \in \mathscr{L}^2(\mu, \h)$ and $\mu$-almost every $\xi \in \T$ we have 
$$(\J \M_{\xi} \J \f)(\xi) = (\M_{\bar \xi} \f)(\xi)$$
and thus,
\begin{equation}\label{Jsjjassayyyy}
 \J\M_\xi\J=\M_{\bar \xi}.
\end{equation}

Proposition \ref{p3.4} below echos a result from \cite[Proposition 4.2]{MR4169409} and relies on the following (easily verified) lemma.
\begin{Lemma}\label{csu}
  Let $U$ be a unitary operator and $C$ be a conjugation on $\mathcal{H}$. Then $UC$ is a conjugation on $\h$ if and only if $C \in \mathscr{C}_{s}(U)$.
\end{Lemma}

\begin{Proposition}\label{p3.4} Let $J$ be a conjugation on $\h$, $\J$ be defined by \eqref{e1.6},  and let $\U\in \mathscr{L}^{\infty}(\mu, \mathcal{B}(\h))$ be such that $\U(\xi)$ is unitary for $\mu$-almost every $\xi \in \T$. Then the following hold. 
\begin{enumerate}
  \item $\M_\U\J$ is a conjugation on $\mathscr{L}^2(\mu, \h)$ if and only if $\U(\xi)$ is $J$--symmetric for $\mu$-almost every $\xi \in \T$.
  \item When $\M_\U\J$ is a conjugation on $\mathscr{L}^2(\mu, \h)$, we have $(\M_\U\J)\M_\xi(\M_\U\J)=\M_{\bar \xi}$.
\end{enumerate}
\end{Proposition}

\begin{proof}
 The proofs that the mapping $\M_{\U}\J$ is antilinear and isometric on the space $\mathscr{L}^2(\mu, \h)$ are straightforward. Assume that $\M_{\U} \J$ is a conjugation. It follows from Lemma \ref{csu} that  for every $\f \in \mathscr{L}^2(\mu, \h)$,
$$ (\M_{\U} \J \f) (\xi) = \U(\xi) (\J \f)(\xi) = \U(\xi) J(\f(\xi))$$
must be equal to
$$(\J \M_{\U^{*}} \f)(\xi) = J((\M_{\U^{*}}\f)(\xi)) = J((\U^{*}(\xi) \f)(\xi)).$$
Therefore, $\U(\xi) J = J \U^{*}(\xi)$ for $\mu$-almost every $\xi \in \T$ and thus, $\U(\xi)$ is $J$-symmetric for $\mu$-almost every $\xi \in \T$. The above argument can be reversed, which proves (a).

To prove (b), observe that for any $\f \in\mathscr{ L}^2(\mu, \h)$ we use the fact that $J \U(\xi) J = \U^{*}(\xi)$ for $\mu$-almost every $\xi \in \T$ to see that 
\begin{align*}
(\M_{\U} \J) \M_{\xi} (\M_{\U} \J \f)(\xi) & = \U(\xi) (\J \M_{\xi} \M_{\U} \J \f)(\xi)\\
& = \U(\xi) J ((\M_{\xi} \M_{\U}\J \f)(\xi))\\
& = \U(\xi) J((\xi \M_{\U} \J \f)(\xi))\\
& = \bar \xi\U(\xi) J((\M_{\U} \J \f)(\xi))\\
& = \bar \xi\U(\xi) J((\J \M_{\U^{*}} \f)(\xi))\\
& = \bar \xi \U(\xi) \U^{*}(\xi) \f(\xi)\\
& = \bar \xi \f(\xi)\\
& = (\M_{\bar{\xi}}\,\f)(\xi).\qedhere
\end{align*}
\end{proof}

\section{Conjugations and bilateral shifts}\label{bilarererer}
Many interesting, and naturally occurring,  unitary operators are bilateral shifts. Examples include the Hilbert and Fourier transforms on $L^2(\R)$;  the translation operator $(U f)(x) = f(x - 1)$ on $L^2(\R)$; the dilation operator $(U f)(x) = \sqrt{2} f(2 x)$ on $L^2(\R)$; and the special class of multiplication operators $U f = \psi f$ on $L^2(m, \T)$ (where $\psi$ is an inner function) which will be the focus of the next section.  The fact that the above operators are bilateral shifts were carefully explained in \cite{JMMPWR_OT28}. This section gives an initial  description of $\mathscr{C}_s(U)$ for this class of operators.  Another description will be discussed in the next section. We begin with precisely what we mean by the term  ``bilateral shift''.

\begin{Definition}\label{lafdhgjdfpsgszhnigbv888}
A unitary operator $U$ on $\h$ is a {\em bilateral shift} if there is a subspace $\mathcal{M} \subset \h$ for which
\begin{enumerate}
\item $U^{n} \mathcal{M} \perp \mathcal{M}$ for all $n \in \Z \setminus \{0\}$;
\item ${\displaystyle \h = \bigoplus_{n = -\infty}^{\infty} U^{n} \mathcal{M}}$.
\end{enumerate}
In the above, note that $U^{-1} = U^{*}$. The subspace $\mathcal{M}$ is called an {\em associated wandering subspace} for the bilateral shift $U$.  Of course there is {\em the} bilateral shift $M_{\xi}$ on $L^2(m, \T)$ defined by $(M_{\xi} f)(\xi)  = \xi f(\xi)$ where the wandering subspace $\mathcal{M}$ is the space of constant functions. 
\end{Definition}

Though the wandering subspace $\mathcal{M}$ in Definition \ref{lafdhgjdfpsgszhnigbv888}  is not unique, its dimension is \cite{MR152896}. The term ``bilateral shift ''  comes from the fact that since
\begin{equation}\label{e15}\mathcal{H}=\bigoplus_{n=-\infty}^{\infty}\, U^n \mathcal{M},
\end{equation}
every $\vec{x} \in \h$ can be uniquely represented as
$$\vec{x} = \sum_{n = -\infty}^{\infty} U^n \vec{x}_n, \; \mbox{where $\vec{x}_n \in \mathcal{M}$ for all $n \in \Z$}.$$
This allows us to define a natural unitary operator
 \begin{equation}\label{WWWWW}
W: \h \to \mathscr{L}^2(m,\m), \quad  W\Big(\bigoplus_{n = -\infty}^{\infty}U^n \vec{x}_n\Big)=\sum_{n = -\infty}^{\infty} \vec{x}_n \xi^n.\end{equation}
 Moreover,
 $WUW^{*}=\M_\xi,$
 where
  $\M_{\xi}$ is the bilateral shift from \eqref{e1.5} defined on $\mathscr{L}^2(m, \mathcal{M})$ by
  $$(\M_{\xi} \f)(\xi) = \xi \f(\xi), \quad   \f(\xi) = \sum_{n = -\infty}^{\infty} \vec{x}_{n} \xi^n \in \mathscr{L}^2(m, \mathcal{M}).$$

  \begin{Example}\label{ljdgflg9osa}
  Examples of bilateral shifts come from a variety of places. These operators were explored in \cite{JMMPWR_OT28} using a Wold-type decomposition. 
  \begin{enumerate}
  \item If $U: L^2(\R) \to L^2(\R)$ is the unitary translation operator defined by $(Uf)(x) = f(x - 1)$ and one sets
  $\mathcal{M} = \chi_{[0, 1]} L^2(\R)$, then $\mathcal{M}$ satisfies conditions (a) and (b) of Definition \ref{lafdhgjdfpsgszhnigbv888} and hence $U$ is unitarily equivalent to $\M_{\xi}$ on $\mathscr{L}^2(m, \mathcal{M})$.
  \item Let $U:L^2(\R) \to L^2(\R)$ be the unitary dilation operator defined by $(Uf)(x) = \sqrt{2} f(2x)$. If
$$\psi(x) = \begin{cases}
\phantom{-}1 & \mbox{for $0 \leq x < \frac{1}{2}$},\\
-1 & \mbox{for $\frac{1}{2} \leq x \leq 1$},\\
\phantom{-}0 & \mbox{otherwise},
\end{cases}$$
then (Haar) wavelet theory \cite{MR1150048} says that the functions
$$\psi_{n, k}(x) := 2^{\frac{n}{2}} \psi(2^{n} x - k), \quad n, k \in \Z,$$
form an orthonormal basis for $L^2(\R)$.
For fixed $\ell \in \Z$, define 
$$\mathcal{W}_{\ell} := \bigvee\{\psi_{\ell, k}: k \in \Z\}.$$
Then  $\mathcal{W}_{\ell} \perp \mathcal{W}_{\ell'}$ for all $\ell, \ell' \in \Z, \ell \not = \ell'$,
and $U \mathcal{W}_{\ell} = \mathcal{W}_{\ell + 1}$ for all $\ell \in \Z$. This means that the subspace $\mathcal{M} = \mathcal{W}_{0}$ satisfies  conditions (a) and (b) of Definition \ref{lafdhgjdfpsgszhnigbv888} and hence $U$ is unitarily equivalent to $\M_{\xi}$ on $\mathscr{L}^2(m, \mathcal{M})$.
\item For an inner function $\psi$ (a bounded analytic function on $\D$ whose radial boundary values are unimodular for $m$-almost every $\xi \in \T$), the operator $M_{\psi} f = \psi f$ is unitary on $L^2 = L^2(m, \T)$ and the model space $\mathcal{M} = \K_{\psi} = H^2 \cap (\psi H^2)^{\perp}$ satisfies the hypothesis of  conditions (a) and (b) of Definition \ref{lafdhgjdfpsgszhnigbv888} \cite[Proposition 5.17]{JMMPWR_OT28}. Hence $M_{\psi}$ is unitarily equivalent to $\M_{\xi}$ on $\mathscr{L}^2(m, \mathcal{M})$. Observe that if $U$ is any  bilateral shift on $\h$ with wandering subspace $\mathcal{N}$ with $\operatorname{dim} \mathcal{N} = N \in \N \cup \{\infty\}$, then the above discussion shows that $U$ is unitarily equivalent to $M_{\psi}$ on $L^2$, where $\psi$ is any inner function whose degree is $N$. When the inner function $\psi$ is a finite Blaschke product with $N$ zeros (repeated according to multiplicity), the {\em degree of $\psi$} is defined to be $N$. In all other cases, the degree of $\psi$ is defined to be $N = \infty$. In the next section, will give a concrete description of $\mathscr{C}_s(M_{\psi})$ written in terms of operators on $L^2$.
  \end{enumerate}
  \end{Example}

For a bilateral shift $U$ on $\h$ we wish to describe $\mathscr{C}_s(U)$.
Since $W U W^{*} = \M_{\xi}$ on $\mathscr{L}^2(m, \mathcal{M})$, where $W: \h \to \mathscr{L}^2(m, \mathcal{M})$ is the unitary operator from \eqref{WWWWW}, Lemma \ref{lem1.2} says that $\CCC=WCW^{*}$  is a conjugation on $\mathscr{L}^2(m,\m)$ such that $\CCC\mathbf{M}_\xi\CCC=\mathbf{M}_{\bar \xi}$.
The following result from  \cite[Thm.~4.8]{MR4169409} describes $\CCC$.

\begin{Theorem}\label{jjHHbbhGHJK}
For a conjugation  $\CCC$ on  $\mathscr{L}^2(m,\m)$,
 the following are equivalent.
 \begin{enumerate}
 \item $\CCC \in \mathscr{C}_{s}(\mathbf{M}_\xi)$;
  \item There is a $\CC\in \mathscr{L}^\infty(m, \mathscr{A}\!\mathcal{B}(\m))$ such that $\CCC=\A_{{\CC}}$ and $\CC(\xi)$ is a conjugation for almost all $\xi\in\mathbb{T}$;
      \item For any conjugation $J$ on $\m$ there is a $\mathbf{U}\in \mathscr{L}^\infty(m,\mathcal{B}(\m))$ such that $\mathbf{U}(\xi)$ is a $J$--symmetric unitary operator for almost all $\xi\in\mathbb{T}$ and $\CCC= \M_{\mathbf{U}}{\mathbf{J}}$.
      \end{enumerate}
\end{Theorem}

This yields a description of $\mathscr{C}_s(U)$ when $U$ is a bilateral shift. Recall the unitary operator $W: \h \to L^2(m, \mathcal{M})$ from \eqref{WWWWW} associated with a bilateral shift with wandering subspace $\mathcal{M}$. 

\begin{Corollary}\label{cococlclcucuc}
Let $U$ be a bilateral shift on $\h$ with an associated wandering subspace $\mathcal{M}$. For a conjugation $C$ on $\h$, the following are equivalent:
\begin{enumerate}
\item $C \in \mathscr{C}_s(U)$;
\item $\CCC = WC W^{*}$ is a conjugation on $\mathscr{L}^2(m, \mathcal{M})$ that satisfies any of the equivalent conditions of Theorem \ref{jjHHbbhGHJK}.
\end{enumerate}
\end{Corollary}

\begin{Example}\label{00o0iJjjIjiiJExampewl666}
For the bilateral shift $M_{\xi}$ on $L^2 = L^2(m, \T)$, we can examine the set $\mathscr{C}_s(M_{\xi})$. Here $\mathcal{M} = \C$ (the constant functions) and $J: \C \to \C$ is $J z = \overline{z}$ and so $\J: L^2 \to L^2$ is $\J f = \overline{f}$.  If $u \in L^{\infty}$ is unimodular  on $\T$, then $M_u f = u f$ is unitary and $\J M_u \J = M_{\bar{u}}$. Theorem \ref{jjHHbbhGHJK} says that any $C \in \mathscr{C}_s(M_{\xi})$  must take the form $C = M_u \J$ for some unimodular $u \in L^{\infty}$. We will see another path to this result in Example \ref{nnnnxnnnxxx2235567796316}.
\end{Example}

\begin{Example}
In a similar way as with the previous example, let  $\mathscr{B} = \{\vec{v}_k\}_{k \geq 1}$ be an orthonormal basis for $\mathcal{M}$. Here $N = \operatorname{dim} \mathcal{M}$ (which might be infinite). Recall the  conjugation $J_{\mathscr{B}}$ on $\mathcal{M}$ defined by $J_{\mathscr{B}} \vec{v}_{k} = \vec{v}_k$ for all $k \geq 1$ (see \eqref{JJJBBBB}).
 Suppose $\U \in \mathscr{L}^{\infty}(m, \mathcal{B}(\mathcal{M}))$ such that $\U(\xi)$ is unitary for almost every $\xi \in \T$. One can check that $[J \U(\xi) J]_{\mathscr{B}} = \overline{[\U(\xi)]_{\mathscr{B}}}$ for almost every $\xi \in \T$. Thus, for $J \U(\xi) J = \U(\xi)^{*}$ it must be the case that $[\U(\xi)^{*}]_{\mathscr{B}} = \overline{[\U(\xi)]_{\mathscr{B}}}$. This means that every conjugation $\CCC$ on $\mathscr{L}^2(m, \mathcal{M})$ for which $\CCC \M_{\xi} \CCC = \M_{\bar{\xi}}$ must take the form
 $(\CCC\f)(\xi) = \U(\xi) J(\f(\xi))$, where $[\U(\xi)^{*}]_{\mathscr{B}} = \overline{[\U(\xi)]_{\mathscr{B}}}$ for almost every $\xi \in \T$. Note that for {\em any} conjugation $J$ on $\mathcal{M}$, we can find an orthonormal basis $\mathscr{B}$ for $\mathcal{M}$ for which $J  = J_{\mathscr{B}}$ \cite[Lemma 1]{MR2187654} . Thus, in a way, the example above is canonical.
\end{Example}

 The results from \cite[Theorem 3.3]{JMMPWR_OT28} show that for any unitary operator  $U$ on $\h$ we have
 $\h=\mathcal{K}\oplus \mathcal{K}^\prime,$
 where $\mathcal{K}$ and $\mathcal{K}'$ reducing subspaces for $U$,
  where $U|_{\mathcal{K}}$ is a bilateral shift and $U|_{\mathcal{K}^\prime}$ has no a bilateral shift part.
For a given conjugation $C$, even assuming that $U$ is $C$-symmetric, the conjugation $C$ cannot be nesseserily decomposed according to this decomposition.
\begin{Example}Consider the same operator as in \cite[Example 5.6]{JMMPWR_OT28}. Namely, let $\h=L^2(\Omega_1)\oplus L^2(\Omega_2)\oplus L^2(\Omega_1)$, where $\Omega_1$, $\Omega_2$ are measurable disjoint subsets of $\mathbb{T}$ such that  $\Omega_1 \cup\Omega_2=\mathbb{T}$
 and
 \[U(f,g,h)(z)=(\xi f(\xi), \xi g(\xi), \xi h(\xi))\quad\text{for}\quad f,h\in L^2(\Omega_1),g\in L^2(\Omega_2). \]
Define the conjugation 
$$C(f,g,h)(\xi)=(\overline{h(\xi)},\overline{g(\xi)},\overline{f(\xi)})\; \; \text{for}\; \;  f,h\in L^2(\Omega_1),g\in L^2(\Omega_2)$$ and observe that $CUC=U^*$. As in \cite[Example 5.6]{JMMPWR_OT28}, the decomposition $\h=\calk\oplus L^2(\Omega_1)$ with $\calk=L^2=L^2(\Omega_1)\oplus L^2(\Omega_2)$, is  a desired  bilateral shift decomposition, but it is not $C$ invariant ($C\calk\nsubseteq \calk$).
\end{Example}

\section{Unitary multiplication operators on $L^2(m, \T)$}\label{S7}

As mentioned in Example \ref{ljdgflg9osa}(c), we have a model for any bilateral shift $U$ on $\h$ as the multiplication operator $M_{\psi}$ on $L^2 = L^2(m, \T)$, where $\psi$ is an inner function whose degree is that of the (uniquely defined) dimension of any wandering subspace for $U$. In this section we give a concrete and tangible description of $\mathscr{C}_s(M_{\psi})$. If $J$ is the conjugation $J f = \overline{f}$ on $L^2$ and $C \in \mathscr{C}_s(M_{\psi})$, then $A = C J$ is a unitary operator on $L^2$ for which $A M_{\psi} = M_{\psi} A$. Thus, in order to describe $\mathscr{C}_{s}(M_{\psi})$, we first need to characterize the bounded operators on $L^2$ that commute with  $M_{\psi}$.


\begin{Remark}\label{s7s7s7s7s7s7s7}
For an inner function $\psi$, a known result (see for example \cite[Proposition 5.17]{JMMPWR_OT28}), says that 
\begin{equation}\label{l2sum}L^2 = \bigoplus _{n = -\infty}^{\infty} \psi^{n} \K_{\psi},\end{equation}
where $\K_{\psi} = H^2 \ominus \psi H^2$ is the model space associated with $\psi$. In other words, $\K_{\psi}$ is a wandering subspace for $M_{\psi}$ (recall Definition \ref{lafdhgjdfpsgszhnigbv888}). Let us set up some notation to be used below. Let $N = \operatorname{dim} \K_{\psi} \in \N \cup \{\infty\}$. Observe that $N$ is finite if and only if $\psi$ is a finite Blaschke product with $N$ zeros, repeated according to multiplicity \cite[Prop.~5.19]{MR3526203}. Also define the space 
$$\bigoplus_{1 \leq j \leq N} L^2 = L^2 \oplus L^2 \oplus  \cdots \oplus L^2.$$
The norm of an element $\f = [f_{j}]_{1 \leq j \leq N}^{t} \in \bigoplus_{1 \leq j \leq N} L^2$ is 
$$\|\f\| = \Big(\sum_{1 \leq j \leq N} \|f_j\|^{2}_{L^2}\Big)^{1/2}.$$
When $N = \infty$, we need to assume that the sum defining $\|\f\|$ is finite. 
Furthermore, the operator $\bigoplus_{1 \leq j \leq N} M_{\xi}$ (called the {\em inflation} of the bilateral shift $M_{\xi}$ on $L^2$) is given by 
$$\Big(\bigoplus_{1 \leq j \leq N} M_{\xi}\Big) \f(\xi) = \xi \f(\xi) = [\xi f_{j}(\xi)]^{t}_{1 \leq j \leq N}.$$
We also define 
$$\ell^{2}_{N} := \Big\{\vec{x} = [x_j]^{t}_{1 \leq j \leq N}, x_j \in \C: \|\vec{x}\|_{\ell^{2}_{N}} = \Big(\sum_{1 \leq j \leq N} |x_j|^2\Big)^{1/2} < \infty \Big\}.$$ When $N = \infty$, note that $\ell^{2}_{N}$ is equal to $\ell^{2}_{+}$ which was the space discussed earlier in \eqref{zoozZZZZellllll}. Finally, observe that 
\begin{equation}\label{kkk*8*ss}
\bigoplus_{1 \leq j \leq N} M_{\xi} \cong \M_{\xi}|_{\mathscr{L}^2(m, \ell^{2}_{N})}
\end{equation}.
\end{Remark}

{
\begin{Theorem}\label{BbnMbBNMmmNB}
For an inner function $\psi$, let 
 $\{h_j\}_{1 \leq j \leq N}$ be an orthonormal basis for the model space $\K_{\psi}$. Then we have the following. 
 \begin{enumerate}
   \item Every $f \in L^2$ has the unique decomposition
\begin{equation}\label{forforf} f = \sum_{1 \leq j \leq N} h_j \cdot (f_j \circ \psi),\end{equation}
where each $f_j$ belongs to $L^2$ and
$\|f\| = (\sum_{1 \leq j \leq N} \|f_j\|^2)^{\frac{1}{2}}$.
   \item The operator
$$W: L^2 \to \bigoplus_{1 \leq j \leq N} L^2, \quad W f = [f_1 \; f_2\;  f_3\; \ldots]^{t},$$ is unitary and
$W^{*} [k_j]^{t}_{1 \leq j \leq N} = \sum_{1 \leq j \leq N}h_j \cdot (k_j \circ \psi).$
   \item $W M_{\psi} W^{*} = \bigoplus_{1 \leq j \leq N} M_{\xi}$ and $W M_{\overline\psi} W^{*} = \bigoplus_{1 \leq j \leq N} M_{\bar\xi}$.
   \item If $A \in \mathcal{B}(L^2)$ satisfies  $A M_{\psi} = M_{\psi} A$, then
$$A f = \sum_{1 \leq j \leq N} (f_j \circ \psi) \sum_{1 \leq k \leq N} h_k \cdot (\phi_{k j} \circ \psi),$$
where $\Phi = [\phi_{i j}]_{1 \leq i, j \leq N}\in\mathscr{L}^{\infty}(m,\mathcal{B}(\ell^2_N))$. 
\item If $\Phi = [\phi_{i j}]_{1 \leq i, j \leq N}\in\mathscr{L}^{\infty}(m,\mathcal{B}(\ell^2_N))$, then the operator $A$ defined in (d) is unitarily equivalent to $M_{\Phi}$ on $\bigoplus_{1 \leq j \leq N} L^2$, i.e., 
$$[f_j]^{t}_{1 \leq j \leq N} \mapsto \Phi [f_j]^{t}_{1 \leq j \leq N}.$$
 \end{enumerate}
\end{Theorem}

\begin{proof}By \eqref{l2sum},
 every $f \in L^2$ can be written (uniquely) as
 \begin{equation}\label{xlcjvlxkcjv8888u99}
f = \sum_{j = -\infty}^{\infty} \psi^{j} \sum_{1 \leq k \leq N} a_{j k} h_k.
\end{equation}
Moreover, since $\psi^j K_{\psi} \perp \psi^{j'} K_{\psi}$ for all $j \not = j'$ and $\{h_k\}_{1 \leq k \leq N}$ forms an orthonormal basis for $\K_{\psi}$, we have
\begin{equation}\label{qppghhhffL}
\|f\|^2 = \sum_{j = -\infty}^{\infty} \sum_{1 \leq k \leq N} |a_{jk}|^2 < \infty.
\end{equation}
Rewrite the expression in \eqref{xlcjvlxkcjv8888u99} as
$$f = \sum_{1 \leq k \leq N} h_k \sum_{j = -\infty}^{\infty} a_{jk} \psi^{j}.$$
By \eqref{qppghhhffL} the function
$f_k = \sum_{j = -\infty}^{\infty} a_{jk} z^{j}$ belongs to $L^2$ and thus,
$$f = \sum_{1 \leq k \leq N} h_k \cdot (f_k \circ \psi).$$ Finally, by Parseval's theorem and \eqref{qppghhhffL}, note that
\begin{align*}
\sum_{1 \leq k \leq N} \|f_k\|^2 & = \sum_{1 \leq k \leq N}  \sum_{j = -\infty}^{\infty} |a_{jk}|^2  = \|f\|^2.
\end{align*}
This verifies statements (a) and (b).  To verify statement (c), note that 
\begin{align*}
W M_{\psi} W^{*} [k_j]^{t}_{1 \leq j \leq N} & = W M_{\psi} \sum_{1 \leq j \leq N} h_j \cdot (k_j \circ \psi)\\
& = W \sum_{1 \leq j \leq N} h_j \cdot  \psi \cdot (k_j \circ \psi)\\
& = W \sum_{1 \leq j \leq N} h_j \cdot ((\xi k_j) \circ \psi)\\
& = [\xi k_j(\xi)]^{t}_{1 \leq j \leq N}.
\end{align*}
This shows that $W M_{\psi} W^{*} = \bigoplus_{1 \leq j \leq N} M_{\xi}$. Similarly, one can verify the second equality in (c).

To prove (d) and (e),  recall from \cite[Chapter III]{MR2003221} that the bounded  operators on the space $\bigoplus_{1 \leq j \leq N} L^2$ that commute with $\bigoplus_{1 \leq j \leq N} M_{\xi}$ must take the form of multiplication by the matrix function $\Phi = [\phi_{ij}]_{1 \leq i, j \leq N} \in \mathscr{L}^{\infty}(m, \ell^2_N)$. 
Putting this all together, we see that if $A \in \mathcal{B}(L^2)$ commutes with $M_{\psi}$,  then $W A W^{*}$ commutes with $\bigoplus_{1 \leq j \leq N} M_{\xi}$ and so $W A W^{*} = \M_{\Phi},$
equivalently $A = W^{*} \M_{\Phi} W$.
This translates to an operator on $L^2$ by 
\begin{align*}
A f & = W^{*} \Phi W f\\
& = W^{*} [\phi_{ij} ]_{1 \leq i, j \leq N}[f_1 \;  f_2 \;  \ldots]^{t}\\
& = W^{*} \Big[\sum_{1 \leq j \leq N} \phi_{1j} f_j \; \;  \sum_{1 \leq j \leq N} \phi_{2j} f_j \; \; \sum_{1 \leq j \leq N} \phi_{3j} f_j \; \; \ldots\Big]^{t}\\
& = h_1\cdot \Big(\sum_{1 \leq j \leq N} \phi_{1j} f_j\Big)\circ \psi + h_2 \cdot \Big(\sum_{1 \leq j \leq N} \phi_{2j} f_j\Big) \circ \psi + \ldots\\
&= (f_{1} \circ \psi) \cdot  \Big(\sum_{1 \leq k \leq N} h_{k} \cdot (\phi_{k1} \circ \psi)\Big) + (f_{2} \circ \psi) \cdot  \Big(\sum_{1 \leq k \leq N} h_{k} \cdot  (\phi_{k2} \circ \psi)\Big) + \cdots,
\end{align*}
which completes the proof of (d) and (e).
\end{proof}

From our earlier discussion, we know that the standard conjugation  $J f = \bar{f}$  on $L^2$ induces the standard conjugation on $\J$ on $ \bigoplus_{1 \leq j \leq N} L^{2}$ by 
$${\J}\vec{F}= [ \bar f_1 \; \bar f_2 \;  \ldots]=\overline{\vec{F}}$$
Here is our description of $\mathscr{C}_{s}(M_{\psi})$ when $\psi$ is inner.

\begin{Theorem}\label{089foidjgigifghjhjh77766}
Suppose that $\psi$ is an  inner function and  $\{h_j\}_{1 \leq j \leq N}$ is an orthonormal basis for $\K_{\psi}$. Then  $C\in\mathscr{C}_s(M_\psi)$ if and only if there is a $\Phi = [\phi_{i j}]_{1 \leq i, j \leq N} \in \mathscr{L}^{\infty}(m, \ell^{2}_{N})$ such that 
\begin{enumerate}
\item $\Phi^{*} \Phi = I$ almost everywhere on $\T$;
\item $\Phi^{t} = \Phi$ almost everywhere on $\T$;
\item 
${\displaystyle C f = \sum_{1 \leq j \leq N} (\bar f_j \circ \psi) \sum_{1 \leq k \leq N} h_k  \cdot (\phi_{k, j} \circ \psi),}$
for all $f \in L^2$ with the decomposition from \eqref{forforf}.
\end{enumerate}
\end{Theorem}

\begin{proof}
Suppose $C \in \mathscr{C}_{s}(M_{\psi})$. Note that $\widetilde{C}:=WCW^*$ defines a  conjugation on  $\bigoplus_{1 \leq j \leq N} L^2$. Since $C M_{\psi} C = M_{\overline{\psi}}$, we can use
Theorem \ref{BbnMbBNMmmNB}(c) to see that 
$$\widetilde{C}\Big( \bigoplus_{1 \leq j \leq N}M_{\xi} \Big)\widetilde{C} = \bigoplus_{1 \leq j \leq N}M_{\bar{\xi}}.$$
Since $\bigoplus_{1 \leq j \leq N}M_{\xi}$ on $\bigoplus_{1 \leq j \leq N}L^2$ is unitary equivalent to $\M_{\xi}$ on $\mathscr{L}^2(m,\ell^2_{N})$ (recall \eqref{kkk*8*ss}), Theorem \ref{jjHHbbhGHJK} yields (with the choice of conjugation $J$ on $\ell^2_{N}$ being $J \vec{x} = \overline{\vec{x}}$), a matrix-valued function  $\Phi = [\phi_{i j}]_{1 \leq i, j \leq N} \in \mathscr{L}^{\infty}(m, \ell^{2}_{N})$  that is unitary valued $m$-almost everywhere, i.e., $\Phi \Phi^{*} = \Phi^{*} \Phi = I$, and such that $\widetilde{C}=\M_\Phi \J$. Moreover, since $\widetilde{C}^2=I$ we see that
$$I = (\M_{\Phi} \J)(\M_{\Phi} \J) = \M_{\Phi} (\J \M_{\Phi} \J) = \M_{\Phi} \M_{\overline{\Phi}}$$ and so 
$\overline{\Phi} = \Phi^{*}$. Combine this with the previous identity to see that $\Phi^{t} = \Phi$. So far we have shown that if $C \in \mathscr{C}_{s}(M_{\psi})$ then $C = W^{*} (\M_{\Phi} \J )W$ where $\Phi$ satisfies the conditions of (a) and (b). If $\Phi$ satisfies the conditions (a) and (b) then a short argument will show that $\Phi$ is unitary valued almost everywhere and the condition (b) will show that $J \Phi J = \Phi^{*}$ almost everywhere. Proposition \ref{p3.4} now says that $\M_{\Phi} \J \in \mathscr{C}_{s}(\M_{\xi})$ and hence $W^{*} (\M_{\Phi} \J) W \in \mathscr{C}_{s}(M_{\psi})$. So we have shown that $C \in \mathscr{C}_{s}(M_{\psi})$ if and only if $C = W^{*} (\M_{\Phi} \J) W$, where $\Phi$ satisfies the conditions in (a) and (b). 

It remains to verify the formula in (c). Observe that for all $f = \sum_{1 \leq j \leq N} h_j  \cdot (f_j \circ \psi)\in L^2$, 
\begin{align*}
C f & = W^{*} \M_\Phi\J W f\\
& = W^{*} [\phi_{ij} ]_{1 \leq i, j \leq N}[\bar f_1 \;  \bar f_2 \; \ldots]^{t}\\
& = W^{*} \Big[\sum_{1 \leq j \leq N} \phi_{1j} \bar f_j \; \; \sum_{1 \leq j \leq N} \phi_{2j} \bar f_j \; \;  \sum_{1 \leq j \leq N} \phi_{3j} \bar f_j \;  \; \ldots\Big]^{t}\\
& = h_1\cdot \Big(\sum_{1 \leq j \leq N} \phi_{1j} \bar f_j\Big)\circ \psi + h_2 \cdot \Big(\sum_{1 \leq j \leq N} \phi_{2j} \bar f_j\Big) \circ \psi + \ldots\\
&= (\bar f_{1} \circ \psi) \cdot  \Big(\sum_{1 \leq k \leq N} h_{k} \cdot (\phi_{k1} \circ \psi)\Big) + (\bar f_{2} \circ \psi) \cdot  \Big(\sum_{1 \leq k \leq N} h_{k}  \cdot (\phi_{k2} \circ \psi)\Big) + \cdots.
\end{align*}

Conversely, suppose there is a $\Phi = [\phi_{i j}]_{1 \leq i, j \leq N} \in \mathscr{L}^{\infty}(m, \ell^{2}_{N})$ such that conditions (a), (b), and (c) hold. Theorem \ref{BbnMbBNMmmNB}(e) and the argument above shows that $C \in \mathscr{C}_{s}(M_{\psi})$. 
\end{proof}

\begin{Remark}
The description of $\mathscr{C}_{s}(M_{\psi})$ for an inner function $\psi$ depends on knowing an orthonormal basis for the model space $\K_{\psi}$. There are several ``natural'' bases one could choose. See \cite[Prop.~5.25]{MR3526203} for some particular examples when $\operatorname{dim} \K_{\psi}$ is finite. 
\end{Remark}

\begin{Example}\label{nnnnxnnnxxx2235567796316}
Suppose $\psi(z) = z$. In this case, $M_{\psi}$ becomes {\em the} bilateral shift $M_{\xi}$ on $L^2$. Furthermore, $\K_{\psi} = \C$ (the constant functions) and the expansion of an $f \in L^2$ from Theorem \ref{BbnMbBNMmmNB} becomes the classical Fourier expansion. Finally, Theorem \ref{BbnMbBNMmmNB} says that every $C \in \mathscr{C}_{s}(M_{\xi})$ must take the form $(C f)(\xi) = u(\xi) \overline{f(\xi)}$ for some $u \in L^{\infty}$ that is unimodular almost everywhere. We observed this by a different method in Example \ref{00o0iJjjIjiiJExampewl666}.
\end{Example}

\begin{Example}
When $\psi(z) = z^2$ above, one can check that the functions $h_{1}(z)  \equiv 1$ and $h_{2}(z) = z$ form an orthonormal basis for  $\K_{\psi}$. Furthermore, using the notation from Theorem \ref{BbnMbBNMmmNB},
\begin{equation}\label{a9re8fugodslkfbsadfssssSDFFGFG}
f_1(\xi) = \sum_{j = -\infty}^{\infty} \widehat{f}(2 j) \xi^j \quad \mbox{and} \quad f_{2}(\xi) = \sum_{j = -\infty}^{\infty} \widehat{f}(2 j + 1) \xi^j.
\end{equation}
 Then, one can check that
$$f(\xi) = h_{1}(\xi) f_{1}(\xi^2) + h_{2}(\xi) f_{2}(\xi^2) = f_{1}(\xi^2) + \xi f_{2}(\xi^2).$$
Theorem \ref{089foidjgigifghjhjh77766} says that every  $C \in \mathscr{C}_s(M_{\xi^2})$ takes the form
$$(C f)(\xi) = \overline{ f_{1}(\xi^2)} (\phi_{11}(\xi^2) + \xi \phi_{21}(\xi^2)) + \overline{ f_{2}(\xi^2)} (\phi_{21}(\xi^2) + \xi \phi_{22}(\xi^2)),$$
where $\phi_{ij}$, $1 \leq i, j \leq 2$, are bounded measurable functions on $\T$ for which $\phi_{12}(\xi) = \phi_{21}(\xi)$ and 
\begin{equation}\label{ex87} \begin{bmatrix}
\overline{ \phi_{11}(\xi)} & \overline{ \phi_{21}(\xi)}\\
\overline{ \phi_{21}(\xi)} & \overline{ \phi_{22}(\xi)}
\end{bmatrix} \begin{bmatrix}
\phi_{11}(\xi) & \phi_{21}(\xi)\\
\phi_{21}(\xi) & \phi_{22}(\xi)
\end{bmatrix}=
\begin{bmatrix}
1 & 0\\
0 & 1
\end{bmatrix}\end{equation}
for almost every $\xi \in \T$. The condition from \eqref{ex87} is equivalent to the identities 
$$|\phi_{11}(\xi)|^2+|\phi_{21}(\xi)|^2= 1,$$
$$ |\phi_{11}(\xi)|=|\phi_{22}(\xi)|,$$
$$\overline{ \phi_{11}(\xi)}\phi_{21}(\xi)+\overline{\phi_{21}(\xi)}\phi_{22}(\xi)=0$$
for almost every $\xi \in \T$.

To make this more tangible, let  $t=\operatorname{Arg}(\xi)\in (-\pi,\pi]$ and let $s(t),\alpha(t)$, $\beta(t)$, and $\gamma(t)$ be any $2\pi$--periodic  bounded real-valued (Lebesgue) measurable functions. Set 
$$
  \phi_{11}(\xi)=e^{i\alpha(t)}s(t),
  $$ 
  $$ \phi_{22}(\xi)=e^{i\beta(t)}s(t),$$
  $$\phi_{21}(\xi) = \phi_{12}(\xi)=e^{i\gamma(t)}\sqrt{1-s^2(t)},$$
and observe that the conditions above yield 
$$\gamma(t)=\tfrac{1}{2}(\pi+\alpha(t)+\beta(t)),$$
which shows that 
$$0 \leq s(t) \leq 1,$$
$$
  \phi_{11}(\xi)=e^{i\alpha(t)}s(t),$$
  $$\phi_{22}(\xi)=e^{i\beta(t)}s(t),$$
  $$\phi_{21}(\xi) = \phi_{12}(\xi) =ie^{\tfrac{i}{2}(\alpha(t)+\beta(t))}\sqrt{1-s^2(t)}.$$
Using the notation from \eqref{a9re8fugodslkfbsadfssssSDFFGFG}, this means that every $C \in \mathscr{C}_s(M_{\xi^2})$ must take  the form
\begin{multline*}
(C f)(\xi) = \overline{ f_{1}(\xi^2)}\Big( e^{i\alpha(2t)}s(2t) + i\xi e^{\tfrac{i}{2}(\alpha(2t)+\beta(2t))}\sqrt{1-s^2(2t)}\Big)\\ + \overline{ f_{2}(\xi^2)}\Big( i\,e^{\tfrac{i}{2}(\alpha(2t)+\beta(2t))}\sqrt{1-s^2(2t)} + \xi e^{i\beta(2t)}s(2t)\Big),\end{multline*}
where $t=\operatorname{Arg}(\xi)\in (-\pi,\pi]$ and $s(t),\alpha(t),\beta(t)$ are any $2\pi$--periodic bounded real-valued  measurable functions.

As a specific nontrivial example we can have the interesting $C \in \mathscr{C}_{s}(M_{\xi^2})$ defined by 
$$(C f)(\xi) = \overline{ f_{1}(\xi^2)} \big(\sin(2t) + \xi \cos(2t)\big) + \overline{ f_{2}(\xi^2)}\big(\cos(2t) - \xi \sin(2t)\big),$$
where $t=\operatorname{Arg}(\xi)$ ($s(t)  = \sin t$, $\alpha(t) \equiv 0$, $\beta(t) \equiv -\pi$)

Another interesting example comes from setting $s(t) \equiv s$, $\alpha(t) = \lambda t$, $\beta(t) = -\pi - \lambda t$, $\lambda \in \R$, which yields 
$$(C f)(\xi) = \overline{f_1(\xi^2)} ( s e^{i \lambda t} + \xi \sqrt{1 - s^2} ) + \overline{f_{2}(\xi)} (\sqrt{1 - s^2} + i  \xi s e^{-i \lambda t}).$$
\end{Example}




\section{Conjugations via the spectral theorem}\label{ST}
This section describes $\mathscr{C}_s(U)$ using the  multiplicity version of the spectral theorem. Important applications of this are Theorem  \ref{new3} and Theorem \ref{LFDKJgnsl;gdhf} below which connect the invariant subspaces of $C \in \mathscr{C}_{s}(U)$ with the hyperinvariant subspaces of $U$.

 We begin with the following (multiplicity) version of the spectral theorem for unitary operators from \cite[Ch. IX, Theorem 10.20]{ConwayFA}. The reader might need a refresher of the notation from \S \ref{section3}.

\begin{Theorem}[Spectral Theorem]\label{spectraltheorem}
For a unitary operator $U$ on a separable Hilbert space $\h$, there are mutually singular measures
$\mu_{\infty}, \mu_1, \mu_2, \ldots \in M_{+}(\T)$, along with Hilbert spaces $\h_{\infty}, \h_{1}, \h_2, \ldots$ each with corresponding  $\operatorname{dim} \h_{k} = k$, $k = \infty, 1, 2, 3, \dots$, and  an isometric  isomorphism
$$\mathcal{I}: \h \to  \mathscr{L}^2(\mu_\infty,\h_\infty)\oplus \mathscr{L}^2(\mu_1,\h_1)\oplus  \mathscr{L}^2(\mu_2,\h_2)\oplus \cdots$$ such that
$\mathcal{I}U \mathcal{I}^{*} $ is equal to the unitary operator 
\begin{equation}\label{MMMMmmmmm}
\M_{\xi}^{(\infty)} \oplus \M_{\xi}^{(1)} \oplus \M_{\xi}^{(2)} \oplus \cdots,
\end{equation}
where for $i = \infty, 1, 2, 3, \ldots$,
$$\M_{\xi}^{(i)}: \mathscr{L}^2(\mu_{i}, \h_{i}) \to  \mathscr{L}^2(\mu_{i}, \h_{i}), \quad \M_{\xi}^{(i)} \f(\xi) = \xi \f(\xi).$$
\end{Theorem}

The main driver of the results of this section is the following.

\begin{Theorem}\label{th1.2}
Let $U$ be a unitary operator on a separable Hilbert space $\h$ with spectral representation as in Theorem \ref{spectraltheorem}. For a conjugation $C$ on $\h$, the following are equivalent.
\begin{enumerate}
  \item $C \in \mathscr{C}_s(U)$,
      \item For each $j=\infty,1,2,\dots$ there are conjugations $\CCC_{j}$ on $\mathscr{L}^{2}(\mu_j,\h_j)$  such that $\CCC_j \in \mathscr{C}_{s}(\M_{\xi}^{(i)})$ and
      \begin{equation*}
        C={\II}^{*}\Big(\bigoplus \CCC_{j}\Big)\II.
      \end{equation*}
      \item For each  $j=\infty,1,2,\dots$ there are $\CC_j \in \mathscr{L}^{\infty}(\mu_j, \mathscr{A}\!\mathcal{B}(\h_j))$  such that $\CC_{j}(\xi)$ is a conjugation for $\mu_{j}$ almost every $\xi \in \T$ and
      \begin{equation*}
        C={\II}^{*}\Big(\bigoplus \A_{\CC_j}\Big)\II.
      \end{equation*}
      \item For each  $i=\infty,1,2,\dots$ and for any conjugation $J_i$ on $\h_i$ there is a  $\U_i \in \mathscr{L}^{\infty}(\mu_i, \mathcal{B}(\h_i))$ such that $\U_i(\xi)$ is unitary  and  $J_i$--symmetric for $\mu_i$-almost every $\xi \in \T$ and
      \begin{equation*}
        C={\II}^{*}\Big(\bigoplus\U_i \J_i \Big)\II={\II}^{*}\Big(\bigoplus\U_i \Big)\Big(\bigoplus \J_i \Big)\II.
      \end{equation*}
\end{enumerate}
\end{Theorem}

\begin{proof}

$(a) \Longrightarrow (d)$:  Set 
\begin{equation}\label{888Uuu88UU88u99}
\widetilde{\M}_{\xi} := \mathcal{I}U\mathcal{I}^{*} = \M_{ \xi}^{(\infty)} \oplus \M_{ \xi}^{(1)} \oplus \M_{ \xi}^{(2)} \oplus \cdots,
\end{equation} and, for a conjugation $C$ on $\h$ define $\widetilde{\!\CC}=\mathcal{I}C\mathcal{I}^{*}$ which, by Lemma \ref{lem1.2}, will be a conjugation on
$$\mathbf{L}^2_{\mathcal{H}} := \mathscr{L}^2(\mu_\infty,\h_\infty)\oplus \mathscr{L}^2(\mu_1,\h_1)\oplus  \mathscr{L}^2(\mu_2,\h_2)\oplus \cdots.$$ Assuming that $C U C = U^{*}$, we see that 
$  \widetilde{\!\CC}\,\widetilde{\M}_\xi \,\widetilde{\!\CC}=\widetilde{\M}_{\bar \xi}.
$

Let $J_i$ be any a conjugation on $\h_i$. Then, as in \eqref{e1.6}, this induces a conjugation $\J_i$ on $\mathscr{L}^2(\mu_i,\h_i)$ defined by
\begin{equation}\label{e1.9}
  (\J_i\f_i)(\xi)=J_i(\f_i(\xi)), \; \;  \f_i\in \mathscr{L}^2(\mu_i,\h_i).
\end{equation}
Define a conjugation\  $\widetilde{\!\J}$ on $\LL^2_{\h}$ by
$$
  \widetilde{\!\J}:=\J_{\infty} \oplus \J_1\oplus \J_2\cdots.
$$
From \eqref{e1.6} we see that\
$
  \widetilde{\!\J}\,\widetilde{\M}_\xi\,\widetilde{\!\J}=\widetilde{\M}_{\bar \xi}.
$
Now observe that
\begin{equation*}
  \widetilde{\M}_\xi (\widetilde{\!\CC}\,\widetilde{\!\J})=(\widetilde{\!\CC}\, \widetilde{\M}_{\bar \xi})\, \widetilde{\!\J}=(\widetilde{\!\CC}\ \widetilde{\!\J}) \widetilde{\M}_\xi.
\end{equation*}
This says that the operator $\widetilde{\!\CC}\ \widetilde{\!\J}$ commutes with $\widetilde{\M_{\xi}}$. 
The spectral theorem applied to $\widetilde{\M_{\xi}}$ also yields the commutant \cite[p.~ 307, Theorem 10.20]{ConwayFA}, namely
 there are operator valued functions
$
  \U_i\in \mathscr{L}^{\infty}(\mu_i,\mathcal{B}(\h_i))$
 for $ i=\infty,1,2,\dots$
such that
\begin{equation*}
  \widetilde{\!\CC}\ \widetilde{\!\J}=\M_{\U_{\infty}}\oplus {\M}_{\U_1}\oplus {\M}_{\U_{2}} \oplus \cdots.
\end{equation*}
Since $\widetilde{\!\CC}\ \widetilde{\!\J}$ is unitary (being linear, isometric, and onto), it follows that each $\widetilde{\M}_{\U_{i}}$ is unitary and, consequently,  $\U_{i}(\xi)$ is unitary for $\mu_i$-almost every $\xi \in \T$.
Therefore,
\begin{equation*}
  \widetilde{\!\CC}=\Big(\bigoplus \M_{\U_i}\Big)\Big(\bigoplus \J_{i}\Big)=\bigoplus (\M_{\U_i}\J_i).
\end{equation*}
Since each\  $\widetilde{\!\CC}|_{\mathscr{L}^2(\mu_i,\h_i)}$ is a conjugation, Proposition \ref{p3.4} says that  the operator $\U_{i}(\xi)$ is $J_{i}$--symmetric $\mu_i$ almost everywhere. Thus, we have verified the implication 
$(a) \Longrightarrow (d)$.

$(d) \Longrightarrow (c)$: For each $i = \infty, 1, 2, \ldots$, take $\CC_{i}=\U_{i}\J_{i}$.

$(c) \implies (b)$: For each $i = \infty, 1, 2, \ldots$, it is enough to take $\CCC_i=\A_{\CC_i}$ and  prove  that
$$\A_{\CC_{i}}\M^{(i)}_{\xi}=\M^{(i)}_{\bar \xi}\A_{\CC_{i}}.$$
Indeed,  for each $\f_{i} \in \mathscr{L}^2(\mu_i, \h_i)$ we have
\begin{align*}
(\A_{\CC_{i}}\M^{(i)}_\xi\f_i)(\xi) & =\CC_{i}(\xi)(\M^{(i)}_{\xi}\f_i)(\xi)\\
&=\CC_{i}(\xi)(\xi \f_i(\xi))\\
&=\bar \xi\CC_{i}(\xi)\f_i(\xi)\\
&= \bar \xi (\A_{\CC_i} \f_i)(\xi)\\
&=(\M^{(i)}_{\bar \xi}\A_{\CC_{i}}\f_i)(\xi).
\end{align*}

$(b) \Longrightarrow (a)$: Recalling the notation from \eqref{888Uuu88UU88u99}, note that
\begin{align*}
CUC & =\II\Big(\bigoplus \CCC_{i}\Big)\Big(\bigoplus \M^{(i)}_\xi\Big)\Big(\bigoplus \CCC_{i}\Big)\II^{*}\\
& =\II\Big(\bigoplus \CCC_{\,i} \M^{(i)}_\xi \CCC_{i}\Big)\II^{*}\\
&=\II\Big(\bigoplus \M^{(i)}_{\bar \xi}\Big)\II^{*}\\
&=\II\widetilde{\M_{\bar \xi}}\II^{*}
=U^*. \qedhere
\end{align*}
\end{proof}
The multiplication operator $\M_\xi$ on $\mathscr{L}^2(\mu,\h)$ is a special case of Theorem \ref{th1.2} -- which we record here for what follows.

\begin{Theorem}\label{new1}
Let $\mu \in M_{+}(\T)$ and $\CCC$ be a conjugation on $\mathscr{L}^2(\mu, \h)$. Then following are equivalent.
\begin{enumerate}
  \item $\CCC\in \mathscr{C}_s(\M_{\xi})$.
    \item
 There are $\CC \in \mathscr{L}^{\infty}(\mu, \mathscr{A}\!\mathcal{B}(\h))$  such that $\CC(\xi)$ is a conjugation for $\mu$ almost every $\xi\in \T$ and
  $\CCC= \A_{\CC} $.
  \item For any conjugation $J$ on $\h$ there is a  $\U \in \mathscr{L}^{\infty}(\mu, \mathcal{B}(\h))$ such that $\U(\xi)$ is  unitary and $J$--symmetric for $\mu$ almost every $\xi \in \T$ and  $ \CCC=\M_\U \J $.
\end{enumerate}
\end{Theorem}
As an application of the above, we have the following connection between conjugations and hyperinvariant subspaces. 

\begin{Proposition}\label{new2}
Let $\mu \in M_{+}(\T)$. If $\mathcal{K}\subset \mathscr{L}^2(\mu, \h)$ is an invariant subspace  for every $\CCC\in \mathscr{C}_s(\M_\xi)$, then $\mathcal{K}$ is hyperinvariant for $\M_\xi$.
\end{Proposition}
\begin{proof} For any fixed conjugation $J$ on $\h$, \eqref{Jsjjassayyyy} says that $\J\in \mathscr{C}_s(\M_\xi)$ and  thus $\mathcal{K}$ is invariant for $\J$. By Theorem \ref{new1},  $\mathcal{K}$ is also invariant for all of the conjugations $\M_u\J$, where $u \in L^\infty(\mu) $ is unimodular $\mu$-almost everywhere.  Therefore,  $\mathcal{K}$ is invariant for every $\M_u=\M_u\J\J$, where  $u\in L^\infty(\mu)$ is unimodular. From here, one can argue that $\K$ is also invariant for $\M_{\chi_\Omega}$ for any Borel set $\Omega\subset \sigma(U)$ (indeed let $u =  1$ on $\Omega$ and $-1$ on $\T \setminus \Omega$ and note that $\chi_{\Omega} = (1 + u)/2$) and, consequently, for any $\M_v$ where $v\in L^\infty(\mu)$.


Now fix any unitary $U_0$ on $\h$ and  let ${\U}_0  \equiv U_0 \in \mathscr{L}^{\infty}(\mu, \mathcal{B(\h)})$ denote the operator-valued constant function. By our discussion in the introduction (see also Proposition \ref{GL}), $U_0$ is $J_0$--symmetric for some conjugation $J_0$. Let $\J_0$ be the conjugation on $\mathscr{L}(\mu,\h)$ given by $(\J_0\f)(\xi)=J_0(\f(\xi))$. Therefore, $\mathcal{K}$ is invariant for $\J_0$ and for $\U_0\J_0$ and thus  for ${\U}_0$.
Similarly, $\K$ is invariant  for $\U_0^*$.  It follows that $\K$ is invariant for the von Neumann algebra containing all constant unitary valued functions in  $\mathcal{B}(\h)$. Finally, $\mathcal{K}$ is invariant   for the  von Neumann algebra generated by $\{\M_{v}: v \in L^{\infty}(\mu)\}$ and the constant operator-valued functions from $\mathscr{L}^{\infty}(\mu, \mathcal{B}(\h))$. From here, one can fashion an argument that $\K$ is invariant for every element of $\mathscr{L}^\infty(\mu, \mathcal{B}(\h))$. Since $\mathscr{L}^\infty(\mu, \mathcal{B}(\h))$ forms the class of operators $\M_{\Phi}$, $\Phi \in \mathscr{L}^\infty(\mu, \mathcal{B}(\h))$, that commute with $\M_{\xi}$,  $\mathcal{K}$ is hyperinvariant for $\M_\xi$.
\end{proof}

From here, we can state our main connection between conjugations and hyperinvariant subspaces.

\begin{Theorem}\label{new3}
Let $U$ be a unitary operator on a separable Hilbert space $\h$.
If $\mathcal{M}\subset \h$  is an invariant subspace  for every  $C\in \mathscr{C}_s(U)$, then $\mathcal{M}$ is hyperinvariant for $U$.
\end{Theorem}
\begin{proof}
The notation from Theorem \ref{th1.2} proves that the invariance of $\mathcal{M}$ for all $C \in \mathscr{C}_{s}(U)$ implies that 
 $\II \mathcal{M}=\bigoplus \mathcal{K}_i$, where the subspace $\mathcal{K}_i$ is invariant for all  all $\M_{\U_i}\J_i$, where $\U_i\in \mathscr{L}^\infty(\mu_i, \h_i)$ is unitary valued $\mu$-almost everywhere. Now apply Proposition \ref{new2} to see that  $\mathcal{K}_i$ is hyperinvariant for each $\M_{\xi}^{(i)}$. Finally, using the description of the commutant of $\M_{\xi}$ \cite[p.~307, Theorem IX.10.20]{ConwayFA}, we obtain that  $\II \mathcal{M}$ is hyperinvariant for $\M_{\xi}$. This shows that $\mathcal{M}$ is a hyperinvariant subspace for $U$.
\end{proof}

The simple example bellow illustrates why requirement in Theorem \ref{new3} that $\mathcal{M}$ is  invariant for {\em every} $C\in \mathscr{C}_s(U)$ is an important one.

\begin{Example} Consider the diagonal unitary operator $U\in \mathcal{B}(\C^2)$ defined by $U=\text{diag}[\lambda ,\lambda]$ (where $\lambda \not = 0$) with respect to the standard basis $\{\vec{e}_1,\vec{e}_2\}$ for $\C^2$. The only hyperinvariant subspaces for $U$ are the trivial ones  $\{\vec{0}\}$ and $\C^2$ (since any $2 \times 2$ matrix commutes with $U$). Define two conjugations $C_1$, $C_2$ on $\C^2$ by 
$$C_1(\vec{e}_1)= \vec{e}_1, \; C_1(\vec{e}_2)= \vec{e}_2 \quad \mbox{and} \quad C_2(\vec{e}_1)=\vec{e}_2, \; C_2(\vec{e}_2) = \vec{e}_1$$ 
(and of course extend antilinearity to all of $\C^2$). One can check that  $C_1,C_2\in \mathscr{C}_s(U)$. Note that each  $C_i$ separately has more  invariant subspaces than the trivial ones $\{\vec{0}\}$ and $\C^2$,  but only the trivial subspaces are simultaneously invariant for both of them.
\end{Example}

Our summary theorem connecting hyperinvariant subspaces and conjugations is the following.

\begin{Theorem}\label{LFDKJgnsl;gdhf}
Let $U$ be a unitary operator  on a separable Hilbert space $\h$ with spectral measure $E(\cdot)$. For a subspace  $\mathcal{M} \subset \h$, the following are equivalent.
\begin{enumerate}
\item $\mathcal{M}$ is hyperinvariant; 
\item $\mathcal{M} = E(\Omega) \h$ for some Borel set $\Omega \subset \sigma(U)$;
\item $\mathcal{M} = \h_{\mu}$ for some $\mu \in M_{+}(\T)$; 
\item $C \mathcal{M} \subset \mathcal{M}$ for every $C \in \mathscr{C}_{s}(U)$.
\end{enumerate}
\end{Theorem}

\begin{proof}
$(a) \Longrightarrow (b)$ is Theorem \ref{iuyrituyeriutyhype}. $(b) \Longrightarrow (c)$ is Theorem \ref{s9dufioiskldfgf}.  $(c) \Longrightarrow (d)$ is Theorem \ref{conj_dec}.  $(d) \Longrightarrow (a)$ is Theorem \ref{new3}.
\end{proof}

\begin{Remark}
One can use Theorem \ref{th1.2} to give an alternate proof of Theorem \ref{kjahfgr4iojtegefeds00} (the description of $\mathscr{C}_{s}(U)$ when $U$ is an $n \times n$ unitary matrix). The same is true for Examples \ref{Foureoirtert} and \ref{Hilertrt}. 
\end{Remark}

\bibliographystyle{plain}

\bibliography{references}

\end{document}